\documentclass[10pt]{article}

\usepackage[english]{babel}
\usepackage[utf8]{inputenc}
\usepackage{amsmath,amsfonts,amssymb,amscd,lscape,amstext,amsthm}	
\usepackage[titletoc]{appendix}
\usepackage{graphicx}
\usepackage[colorinlistoftodos]{todonotes}
\usepackage[final, colorlinks = true, 
            linkcolor = black, 
            citecolor = black]{hyperref} 

\newtheorem{lemma}{Lemma}[section]
\newtheorem{proposition}[lemma]{Proposition}
\newtheorem{theorem}[lemma]{Theorem}
\newtheorem{corollary}[lemma]{Corollary}
\newtheorem{definition}{Definition}[section]
\newtheorem{remark}[lemma]{Remark}

\newcommand {\R} {\mathbb{R}} 
 \newcommand {\N} {\mathbb{N}}

\newcommand{\CC}{\mathcal{C}}
\newcommand{\HH}{\mathcal{H}}
\newcommand{\la}{\langle}
\newcommand{\ra}{\rangle}
\newcommand{\diff}{\hspace{0.3em}\mathrm{d}}
\newcommand {\p} {\partial}
\newcommand{\rr}{\bar{r}}
\newcommand {\va} {\varphi}
\newcommand{\ep}{\varepsilon}
\newcommand{\KK}{\tilde K}

\title{Lipschitz stability estimate for the simultaneous recovery of two coefficients in the anisotropic Schr\"odinger type equation via local Cauchy data}
\author{Sonia Foschiatti
\footnote{Dipartimento di Matematica e Geoscienze, via Valerio 12/1, Universit\`{a} di Trieste, Trieste, 34126, Italy. e-mail:sonia.foschiatti@phd.units.it}}
\date{ }

\begin{document}

\maketitle

\begin{abstract}
    We consider the inverse problem of the simultaneous identification of the coefficients $\sigma$ and $q$ of the equation div$(\sigma\nabla u) + qu=0$ from the knowledge of the complete Cauchy data pairs. We assume that $\sigma=\gamma A$ where $A$ is a given matrix function and $\gamma, q$ are unknown piecewise affine scalar functions. No sign, nor spectrum condition on $q$  is assumed. We derive a result of global Lipschitz stability in dimension $n\geq 3$. The proof relies on the method of singular solutions and on the quantitative estimates of unique continuation.
\end{abstract}

\\


\section{Introduction}

In this paper we deal with the inverse problem of the \emph{coefficient identification} in a Schr\"odinger type equation. Let $\Omega$ be a bounded domain of $\R^n$ and $\Sigma$ be a non-empty portion of the boundary $\p\Omega$. Let us denote by $u\in H^1(\Omega)$ a weak solution to the equation

 \begin{equation}\label{eqn: main bvp}
         \,\mbox{div}(\sigma \nabla u) + qu = 0, \qquad \mbox{in }\Omega.
\end{equation}
We denote with $u|_{\Sigma}$ the trace of the solution $u$ to \eqref{eqn: main bvp} at the open portion $\Sigma$ and $\sigma\nabla u\cdot \nu|_{\Sigma}$ the trace of the conormal derivative of $u$ at $\Sigma$. Here $\nu$ is the outward unit normal of $\p\Omega$, which is well-defined under appropriate assumptions of regularity at the boundary.
The inverse problem consists in the simultaneous determination of the pair of coefficients $\sigma$ and $q$ from the knowledge of all the possible pairs of Cauchy data $(u|_{\Sigma}, \sigma \nabla u\cdot \nu|_{\Sigma})$ on the open portion $\Sigma$ of the boundary.

The boundary value problem associated to \eqref{eqn: main bvp} comprises a large class of inverse problems that are characterised by their ill-posed nature. Let us briefly review a selected collection that has motivated our interest in the stability issue for \eqref{eqn: main bvp}. The inverse problem of recovering only the coefficient $\sigma$ when $q=0$ from the knowledge of the Dirichlet to Neumann map is known as the Calder\'on problem, which was first introduced by A. Calder\'on in \cite{Calderon1980}. The uniqueness issue has been treated by Sylvester and Uhlmann in \cite{Uhlmann1987} for conductivities of class $C^2$ (see \cite{Uhlmann2014} for a complete survey). The stability issue was first investigated by Alessandrini in \cite{Alessandrini1988} for isotropic conductivities belonging to $H^s(\Omega)$ for $s>\frac{n}{2}+2$. Under these assumptions, the author proved a stability estimate with logarithmic modulus of continuity. Later Mandache \cite{Mandache2001} demonstrated that this estimate is indeed optimal under very general hypothesis. The ill-posed character in the inverse conductivity problem is a common denominator in this field and it constitutes an obstruction in numerical reconstructions. To reduce the ill-posed nature, it is convenient to restrict the space of admissible conductivities by imposing appropriate \emph{a-priori} assumptions on the conductivity. In \cite{Alessandrini2005}, Alessandrini and Vessella demonstrated a Lipschitz stability estimate for piecewise constant conductivities defined on a finite partition of the domain $\Omega$ that satisfy some \emph{a-priori} bounds. Rondi \cite{Rondi2003} has proved that the Lipschitz constant appearing in the stability estimate \cite[Theorem 2.7]{Alessandrini2005} behaves exponentially with respect to the number $N$ of subdomains of the partition. This result was subsequently extended by Di Cristo and Rondi \cite{Dicristo2003} for the inverse scattering problem and by Sincich \cite{Sincich2007} for the corrosion detection problem. Recently, in \cite{Alberti2023}, Alberti et al. have extended these ideas by proving that for coefficients belonging to finite dimensional manifolds, uniqueness and stability are guaranteed. In this direction, Lipschitz stability estimates have been proved for real and complex finite dimensional isotropic coefficients (\cite{Alessandrini2005, Alessandrini2017, Beretta2011}), for a special type of anisotropic conductivities (\cite{Gaburro2015, Foschiatti2021}), for polyhedral inclusions in a conductive medium (\cite{Beretta2021, aspri2022, Beretta2022}), for the non local operator (\cite{Ruland2019}) and for the elasticity case (\cite{Eberle2021}). As a disclaimer, we would like to remark that this list is far to be a complete collection of stability results that have been proven in the last decades. However, we would like to underline the fact that these results are based on the singular solution method and unique continuation techniques.

When $\sigma$ is the identity matrix, equation \eqref{eqn: main bvp} is the Schr\"odinger equation. Lipschitz stability has been demonstrated both when the Dirichlet to Neumann map is defined (hence under suitable spectral conditions) and when only Cauchy data are available, in the case of finite dimensional potential $q$ (see \cite{Beretta2013, Alessandrini2018, Ruland2022}). When $q$ has positive sign, \eqref{eqn: main bvp} is the reduced wave equation or the Helmholtz equation. In \cite{Beretta2016} the authors succeeded in proving the conditional Lipschitz stability at selected frequencies, using the Dirichlet to Neumann map. See also \cite{Alessandrini2019} for the related numerical experiments.  

When $q$ is a non-positive scalar function, the boundary value problem associated to \eqref{eqn: main bvp} models the propagation of light in a body and corresponds to the diffusion approximation of the radiative transfer equation in the frequency domain. In this framework, the coefficients $\sigma$ and $q$ model the diffusive and absorption coefficients, respectively.  The corresponding application is the diffusive optical tomography (DOT), a novel, non-invasive technique that allows one to map the optical properties of a tissue (see \cite{Arridge1999, Arridge2009, Applegate2020}). In \cite{Arridge1998}, Arridge and Lionheart demonstrated that, under generic assumptions, it is not possible to simultaneously recover diffusion and absorption coefficients. However, later results showed that if the coefficients belong to a finite dimensional space of bounded functions, it is possible to determine the coefficients simultaneously. In \cite{Harrach2009}, Harrach proved uniqueness under the assumption that the diffusion coefficient is piecewise constant and the absorption coefficient is piecewise analytic. The author used the technique of localised potentials, developed by the same author in \cite{Gebauer2008}, and monotonicity method also used in \cite{Harrach2012}. Recently, the method of localised potentials was successfully employed by Harrach and Lin (\cite{harrach2023}) to recover piecewise analytic coefficients in a semilinear elliptic equation, under proper hypotheses that ensure the existence of the Dirichlet to Neumann map. 

The aim of this work is to demonstrate a Lipschitz stability estimate that holds simultaneously for both the coefficients. Lipschitz stability is derived by using a constructive approach based on the singular solution method and the quantitative estimates of unique continuation (see \cite{Alessandrini2005, Bellassoued2007, Alessandrini2018} and \cite{vessella2023} for a recent survey). We consider a partition $\{D_i\}_{i=1}^N$, $N\in \N$ of the domain $\Omega$ consisting in a finite number of bounded domains with boundary of class $C^2$. Notice that in previous works the boundary regularity was at most Lipschitz. In our context we need to impose a higher boundary regularity because we require that the Cauchy problem \eqref{eqn: green system} is well-posed and that a suitable version of the inequality of the three spheres proved in \cite{Carstea2020} can be applied.

We consider finite-dimensional coefficients of the form
\[
\sigma(x) := \gamma(x) A(x) = \left(\sum_{j=1}^N \gamma_{j}(x) \chi_{D_j}(x) \right)A(x), \qquad q(x) := \sum_{j=1}^N q_{j}(x) \chi_{D_j}(x),
\]
where $\gamma_j, q_j$ are piecewise affine functions for $j=1,\dots,N$, and $A(x)$ is a known $C^{1,1}(\Omega,Sym_n)$ matrix function, with $Sym_n$ the space of $n\times n$ real symmetric matrices. For simplicity, we denote with $\CC_i$ the local Cauchy data set associated to the pairs of coefficients $\{\sigma^{(i)}, q^{(i)}\}_{i=1,2}$ and $d(\CC_1,\CC_2)$ denotes the distance between the sets (see equation \eqref{eqn: distance}). In Theorem \ref{stability} we prove that
\begin{equation}\label{eqn: stability}
    \|\sigma^{(1)}-\sigma^{(2)}\|_{L^{\infty}(\Omega)}+\|q^{(1)}-q^{(2)}\|_{L^{\infty}(\Omega)} \leq C d(\CC_1, \CC_2),
\end{equation}
with $C>0$ a constant that depends only on the a priori data.

Our stability result is based on the method of singular solutions, whose application in the study of the stability in inverse problems dates back to \cite{Isakov1988, Alessandrini1988}. 

Recall that the boundary value problem problem associated to \eqref{eqn: main bvp} may be in the eigenvalue regime, so no unique solution is guaranteed. In \cite[Lemma 4.1]{Foschiatti2023} the authors constructed Green's functions for a boundary value problem with prescribed complex-valued Robin data on a portion $\Sigma_0$ of the boundary (see also \cite{Alessandrini2018}).

As in \cite{Alessandrini1990, Alessandrini2017, Alessandrini2018}, we start the analysis by providing a boundary stability estimate of H\"older type for both the coefficients $\gamma$ and $q$ of the form
\begin{equation}\label{eqn: boundary}
    \|\sigma^{(1)}-\sigma^{(2)}\|_{L^{\infty}(\Sigma)} + \|q^{(1)}-q^{(2)}\|_{L^{\infty}(\Sigma)} \leq C(d(\CC_1,\CC_2) + E)^{1-\eta} d(\CC_1,\CC_2)^{\eta},
\end{equation}
for $0<\eta<1$, $E=\max\{\|\sigma^{(1)}-\sigma^{(2)}\|_{L^{\infty}(\Omega)}, \|q^{(1)}-q^{(2)}\|_{L^{\infty}(\Omega)} \}$, and $C>0$ is a positive constant depending on the a-priori data only. Estimate \eqref{eqn: boundary} is derived by applying an Alessandrini's type argument, and the study of the blowup rate of the Green's function near the discontinuity interface. 

Despite the H\"older boundary estimate, combining the a-priori assumptions on $\gamma$ and $q$ yields a Lipschitz estimate that holds in the interior of the domain, an estimate that is derived by applying an iterative procedure that we now describe.

We fix  a chain of subdomains ${D_0, D_1, \dots, D_K}$ of the partition of $\Omega$ that, up to a reordering of indices, are contiguous. The chain connects $D_0$ to the domain $D_K$ where $E$ is achieved. We adopt and generalise the iterative strategy introduced in \cite{Alessandrini2005} for the determination of one parameter as follows. First, we determine the following H\"older type estimate for the two coefficients in $D_1$ in terms of the Cauchy data:
\[
\|\gamma^{(1)} -\gamma^{(2)}\|_{L^{\infty}(D_1)} + \|q^{(1)} - q^{(2)}\|_{L^{\infty}(D_1)} \leq C (E+\ep) \Big(\frac{\ep}{E + \ep} \Big)^{\tilde \eta_1},
\]
where $0<\tilde\eta_1 <1$ depends on the a-priori data only. Then, we use a two step procedure. 
\begin{itemize}
    \item[i)] First, we determine an upper bound for $\|\gamma^{(1)}_2-\gamma^{(2)}_2\|_{L^{\infty}(D_{2})}$ using the asymptotic estimates of the singular solutions near the discontinuous interface and the a-priori information on the $q^{(i)}$, for $i=1,2$. We also make use of quantitative estimates of unique continuation (see Proposition \ref{proposizione unique continuation finale}), based on propagation of smallness estimates demonstrated by Carstea and Wang in \cite{Carstea2020} that hold for piecewise Lipschitz coefficients. 
    \item[ii)] Second, we estimate $\|q^{(1)}_2-q^{(2)}_2\|_{L^{\infty}(D_{2})}$ by taking advantage of the stability estimate in $i)$, the asymptotic estimates for the Green functions and the quantitative estimates of unique continuation. 
\end{itemize}

Proceeding iteratively along the chain of subdomains up to $D_K$, we derive the following inequality:
\begin{equation}\label{eqn: omegak}
\|\gamma^{(1)} -\gamma^{(2)}\|_{L^{\infty}(D_K)} + \|q^{(1)} -q^{(2)} \|_{L^{\infty}(D_K)} \leq C (E+\ep) \omega^{(3(K-1))}_{\tilde \eta_K}\Big(\frac{\ep}{E + \ep} \Big),
\end{equation}
with $0<\tilde \eta_K<1$ a constant that depends on the a-priori data only, and $\omega_{\tilde \eta_K}$ is a modulus of continuity of logarithmic type of the form
\[
\omega_{\tilde \eta_K}(t) \leq c |\ln t^{-1}|^{-\tilde\eta_K}, \quad \text{for }t\in(0,1).
\]
The Lipschitz stability estimate \eqref{eqn: stability} is deduced by \eqref{eqn: omegak} and the fact that $\omega_{\tilde\eta_K}$ is invertible.

The article is organized as follows. In Section \ref{sec: 2} we introduce the a priori assumptions on the domain $\Omega$ and the coefficients $\sigma, q$. After defining the local Cauchy data, we state the stability result (Theorem \ref{stability}) and Corollary \ref{corollary}. In Section \ref{sec: 3} we introduce the main tools needed to prove the theorem, namely the asymptotic estimates for the Green's function near the discontinuity interface (Proposition \ref{prop: asymptotic estimates}) and the quantitative estimates of the unique continuation (Proposition \ref{proposizione unique continuation finale}). In Section \ref{sec: 4} we prove the Lipschitz stability estimate (Theorem \ref{stability}). In Section \ref{sec5} we give a sketch of the proofs the technical Propositions introduced in Section \ref{sec: 3}. In the Appendix we demonstrate the stability at the boundary for both the coefficients $\sigma$ and $q$. 

\section{Notation and Main result}\label{sec: 2}

In this section we recall the main definitions and summarise the \emph{a-priori} information concerning the domain $\Omega$ and the coefficients $\sigma, q$ of \eqref{eqn: main bvp}. Then we state the Lipschitz stability estimate (Theorem \ref{stability}).

For a point $x=(x_1, x_2,\dots,x_n)^t \in \mathbb{R}^n$, we can write $x=(x',x_n)$, where $x'\in\mathbb{R}^{n-1}$ and $x_n\in\mathbb{R}$. For any $x\in \mathbb{R}^n$, $B_r(x), B_r'(x')$ denote the open balls in $\mathbb{R}^{n},\mathbb{R}^{n-1}$ centred in $x$ and $x'$ respectively with radius $r>0$. Let $B_r=B_r(0)$ and $B_r'=B_r'(0)$. Let $\R^n_{\pm} = \{ (x',x_n)\in \R^{n-1}\times \R \,:\, x_n\gtrless 0 \}$ the positive (negative) real half-space, $B^{\pm}_r = B_r\cap \R^n_{\pm}$ the positive (negative) half-ball centred in the origin.

\begin{definition}\label{def: C^2 boundary}
    We say that $\Omega\subset\R^n$ has the \emph{boundary of class $C^2$ with constants $r_0, L>0$} if for any point $P\in\p\Omega$ there exists a rigid transformation under which $P$ coincides with the origin $0$ and
    \[
    \Omega\cap B_{r_0}=\left\{x\in B_{r_0}\,:\,x_n>\varphi(x')\right\},
    \]
    where $\varphi$ is a $C^2$ function on $\overline{B'_{r_0}}$ such that 
    \[
    \varphi(0)=|\nabla\varphi(0)|=0 \quad \text{and} \quad \|\varphi\|_{C^{2}(B'_{r_0})}\leq L r_0,
    \]
    with
    \[
    \|\varphi\|_{C^2(\overline{B'_{r_0}})} = \sum_{|\beta|\leq 2} r_0^{|\beta|} \|\p^{\beta} \varphi\|_{L^{\infty}(B'_{r_0})},
    \]
    with $\beta=(\beta_1,\dots,\beta_n)\in \N_0^n$, $\p^{\beta}=\p_{x_1}^{\beta_1}\p_{x_2}^{\beta_2}\dots \p_{x_n}^{\beta_n}$.
\end{definition}
	
\begin{definition}\label{def: flat portion}
    Let $\Omega\subset \R^n$ be a bounded domain. A boundary portion $\Sigma$ of $\partial\Omega$ is said to be a \emph{flat portion} of size $r_0>0$ if for each point $P\in\Sigma$ there exists a rigid transformation under which $P$ coincides with the origin $0$ and
    \begin{equation*}
            \Sigma\cap B_{r_{0}} = \{x\in B_{r_0}\,:\,x_n=0\},\qquad
            \Omega\cap B_{r_{0}} =\{x\in B_{r_0}\,:\,x_n>0\}.
    \end{equation*}
\end{definition}
	
\subsection{A priori information on the domain}\label{sec: domain}

Consider $\Omega\subset \R^n$ a bounded, measurable domain with boundary $\p\Omega$ of class $C^2$ with constants $r_0, L$ such that
\begin{equation}\label{eqn: misura omega}
    |\Omega|\leq C r_0^n,
\end{equation}
where $|\Omega|$ denotes the Lebesgue measure of $\Omega$ and $C$ is a positive constant. Let $\Sigma \subset \p\Omega$ be a flat portion of size $r_0$. We assume that there exists a partition of bounded domains $\{D_m\}_{m=1}^N$, $N\in \N$, $N>1$, contained in $\Omega$ that satisfies the following conditions:
\begin{itemize}
    \item[(D1)] Each $D_m$ for $m=1,\dots,N$ is connected with boundaries $\p D_m$ of class $C^2$ with constants $r_0$, $L$. These domains are pairwise non-overlapping.
    \item[(D2)] $\displaystyle \overline{\Omega} = \bigcup_{m=1}^{N}\overline{D}_m$.
    \item[(D3)] There exists a region, denoted by $D_1$, such that the intersection $\partial{D}_1\cap\Sigma$ contains a flat portion $\Sigma_1$ of size $r_0\slash 3$.  For any index $m\in\{2,\dots , N\}$ we assume that the intersection
    \[\partial{D}_{m}\cap \partial{D}_{m+1}\]
    contains a flat portion $\Sigma_{m+1}$ of size $r_0\slash 3$ such that $\Sigma_{m+1}\subset\Omega$. Furthermore, we assume that there exist a point $P_{m+1}\in\Sigma_{m+1}$ and a rigid transformation under which $P_{m+1}$ coincides with the origin $0$ and
    \begin{equation*}
	\begin{split}
			\Sigma_{m+1}\cap B_{\frac{r_0}{3}} 
				=\Big\{x\in B_{\frac{r_0}{3}}\,:\,x_n=0\Big\}, \qquad
				D_{m}\cap B_{\frac{r_0}{3}} 
				=\left\{x\in B_{\frac{r_0}{3}}\,:\,x_n<0\right\}.
			\end{split}
		\end{equation*}
    Notice that since the boundary is of class $C^2$, for each pair of contiguous subdomains one can simply consider a local diffeomorphism that flattens the boundary. However, in view of proving the stability estimate, it is convenient to give such assumption for granted.
\end{itemize}

\subsection{A priori information on the coefficients}\label{sec: coefficient}	

Consider the elliptic equation
\begin{equation}\label{eqn: main eq}
    \mbox{div}(\sigma \nabla u) + q u =0 \quad \mbox{in }\Omega.
\end{equation}
The coefficient $\sigma$ is a bounded, measurable $n\times n$ real matrix function of the form
\begin{equation}
   \sigma(x) = \gamma(x)\,A(x), \qquad x\in \Omega,
\end{equation}
that satisfies the following conditions:
\begin{itemize}
  \item[(C1)]  The scalar function $\gamma$ is piecewise affine and has the form
    \begin{align*}
        \gamma(x) = \sum_{j=1}^N \gamma_j(x) \chi_{D_j}(x), \qquad \gamma_j(x) = a_j + b_j\cdot x,\qquad x\in \Omega
    \end{align*}
    for $a_j\in \R$, $b_j\in \R^n$ and $D_j$ for $j=1,\dots,N$ are the given subdomains of the partition as in Section \ref{sec: domain}. Moreover, there exists a constant $\bar{\gamma}>1$ such that for a.e. $x\in\Omega$,
    \begin{equation}
        \bar{\gamma}^{-1} \leq \gamma(x) \leq \bar{\gamma}, \qquad \text{for any }j=1,\dots,N.
    \end{equation}
    \item[(C3)] The matrix function $A$ belongs to the space $C^{1,1}(\Omega,Sym_n)$ and there is a constant $\bar{A}>0$ such that
    \begin{equation}\label{Abar}
        \|a_{ij}\|_{C^{1,1}(\Omega)} \leq \bar{A}\qquad \text{for }i,j=1,\dots,n,
    \end{equation}
    where
    \[
    \|a_{ij}\|_{C^{1,1}(\Omega)} = \|a_{ij}\|_{C^1(\Omega)} + r_0^2 \sup_{x,y\in \Omega, x\neq y} \frac{|a_{ij}(x) - a_{ij}(y)|}{|x-y|}.
    \]
    \item[(C4)] (\emph{Uniform ellipticity condition)} There exists a constant $\bar{\lambda}>1$ such that
    \begin{equation}
        \bar{\lambda}^{-1} |\xi|^2 \leq A(x)\xi\cdot\xi \leq \bar{\lambda} |\xi|^2 \qquad \mbox{for any }\xi\in\R^n, \mbox{ for a.e. }x\in\Omega.
    \end{equation}
    \item[(C5)] The coefficient $q\in L^{\infty}(\Omega)$ is a piecewise affine function of the form
    \begin{equation*}
        q(x) = \sum_{j=1}^N q_j(x) \chi_{D_j}(x),\qquad q_j(x) = c_j + d_j\cdot x,\qquad x\in\Omega,
    \end{equation*}
    for $c_j\in \R$, $d_j\in \R^n$ and $D_j$ for $j=1,\dots,N$ are the given subdomains of the partition as in Section \ref{sec: domain}. 
    \item[(C6)] There are $\bar{\sigma}, \bar{q} >0$ such that
    \begin{equation}
        \|\sigma\|_{L^{\infty}(\Omega)} \leq \bar{\sigma}, \qquad \|q\|_{L^{\infty}(\Omega)} \leq \bar{q}.
    \end{equation}
    \end{itemize}

The collection of constants $\{r_0, L, N, \bar{\lambda}, \bar{\gamma}, \bar{\sigma}, \bar{q} \}$ along with the dimension $n\geq 3$ are called the \emph{a priori data}. We would like to remark here that we decide to follow the so-called \emph{constant variable convention} that consists in denoting with the letter $C$ positive constants that depend on the a priori data only and that may vary from line to line in the inequalities.

\begin{remark}
  The class of functions $\gamma(x), q(x)$ form a finite dimensional linear subspace. The $L^{\infty}$ norms of $\gamma, q$ are equivalent to the following norms:
  \begin{align*}
    |||\gamma||| = \max\limits_{j=1,\dots,N} \{|a_j| + |b_j|\},\qquad
    |||q||| = \max\limits_{j=1,\dots,N} \{|c_j| + |d_j|\},
  \end{align*}
  modulo some constants depending on the a priori data.
\end{remark}

\subsection{Local Cauchy data set}

Before describing the local Cauchy data, we recall the definition of some useful trace spaces. Let $H^{\frac{1}{2}}_{co}(\Sigma)$ be the trace space of functions having compact support in $\Sigma$. The space $H^{\frac{1}{2}}_{00}(\Sigma)$ is the closure of $H^{\frac{1}{2}}_{co}(\Sigma)$ under the norm $H^{\frac{1}{2}}(\p\Omega)$. The distributional space $H^{-\frac{1}{2}}(\p\Omega)|_{\Sigma}$ is the restriction of the trace space of distributions $H^{-\frac{1}{2}}(\p\Omega)$ to $\Sigma$. 

For $f\in H^{\frac{1}{2}}_{00}(\Sigma)$, the boundary value problem
\begin{equation}
    \begin{cases}
      \mbox{div}(\sigma \nabla u) + q u =0, &\mbox{in }\Omega,\\
      u = f, &\mbox{on }\p\Omega,
   \end{cases}
\end{equation}
may have no unique solution or any solution in the natural Sobolev space, since we do not make no assumption on the sign of $q$. In this general framework, the Dirichlet-to-Neumann map may not be defined.
As in \cite{Alessandrini2018} (see also \cite[§ 5, pag. 152]{Isakov2017}), we find convenient to introduce a set to model the pairs $(u|_{\p\Omega}, \sigma \nabla u\cdot \nu|_{\p\Omega})$.
\begin{definition}\label{def: localcauchy}
    The \emph{local Cauchy data} associated to $\sigma, q$ having zero first component on $\p\Omega\setminus \overline{\Sigma}$ is the set
    \begin{align*}
    \CC_{\sigma,q}(\Sigma) = \Big\{ (f,g)&\in H^{\frac{1}{2}}_{00}(\Sigma)\times H^{-\frac{1}{2}}(\p\Omega)|_{\Sigma}\,:\,\mbox{there exists }u\in H^1(\Omega)\,\text{such that }\\
    &\mbox{div}(\sigma\nabla u) +qu=0\quad\mbox{in }\Omega,\\
    &u|_{\p\Omega} = f,\\
    &\la \sigma\nabla u\cdot \nu|_{\p\Omega}, \va \ra = \la g,\va \ra\quad \mbox{for any }\va\in H^{\frac{1}{2}}_{00}(\Sigma)\Big\}.
    \end{align*}
\end{definition}
Notice that $\CC_{\sigma,q}(\Sigma)$ is a subset of $H^{\frac{1}{2}}_{00}(\Sigma)\times H^{-\frac{1}{2}}(\p\Omega)|_{\Sigma}$, which is a Hilbert space with the norm
\[
\|(f,g)\|_{H^{\frac{1}{2}}_{00}(\Sigma)\bigoplus H^{-\frac{1}{2}}(\p\Omega)|_{\Sigma}} = \left( \|f\|_{H^{\frac{1}{2}}_{00}(\Sigma)}^2 + \|g\|_{H^{-\frac{1}{2}}(\p\Omega)|_{\Sigma}}^2\right)^{\frac{1}{2}}.
\]
If $S_1, S_2$ are two closed subspace of a given Hilbert space $\HH$, the distance between $S_1$ and $S_2$ is defined as 
\begin{equation}
    d(S_1,S_2) = \max\left\{\inf_{h\in S_2 \setminus \{0\}} \sup_{k\in S_1} \frac{\|h-k\|_{\HH}}{\|h\|_{\HH}}, \inf_{k\in S_1 \setminus \{0\}} \sup_{h\in S_2} \frac{\|h-k\|_{\HH}}{\|k\|_{\HH}} \right\}.
\end{equation}
If $d(S_1,S_2)<1$ then
\[
d(S_1,S_2) = \inf_{h\in S_2 \setminus \{0\}} \sup_{k\in S_1} \frac{\|h-k\|_{\HH}}{\|h\|_{\HH}},
\]
(see \cite{Knyazev2010} and \cite{Alessandrini2018}). In our framework, for two pairs of coefficients $\{\sigma^{(k)}, q^{(k)}\}$ with $k=1,2$, the corresponding local Cauchy data are the sets $\CC_1=\CC_{\sigma^{(1)},q^{(1)}}(\Sigma)$, $\CC_2=\CC_{\sigma^{(2)}, q^{(2)}}(\Sigma)$. Since we are interested in the occurrence when $\CC_1$ and $\CC_2$ are rather close, the distance is given by
\begin{equation}\label{eqn: distance}
d(\CC_1,\CC_2) = \inf_{(f_2,g_2)\in \CC_2 \setminus \{(0,0)\}} \sup_{(f_1,g_1)\in \CC_1} \frac{\|(f_2,g_2)-(f_1,g_1)\|_{\HH}}{\|(f_2,g_2)\|_{\HH}},
\end{equation}
with $\HH=H^{\frac{1}{2}}_{00}(\Sigma)\oplus H^{-\frac{1}{2}}(\p\Omega)|_{\Sigma}$. Notice that the above distance is computed between the closures $\overline{\CC}_1$ and $\overline{\CC}_2$ and it can be shown that the Cauchy data $\CC_1, \CC_2$ are indeed closed. Moreover, if the direct problem is well-posed, then the local Cauchy data represents the graph of the local Dirichlet to Neumann map.

We state the stability estimate, the proof of which is deferred to Section \ref{sec: 4}.

\begin{theorem}\label{stability}
    Let $\Omega\subset\R^n$, $\Sigma\subset\p\Omega$ be a bounded domain and a non-empty portion as in Section \ref{sec: domain}. Let $\{\sigma^{(i)}, q^{(i)}\}$ for $i=1,2$ be two pairs of parameters that satisfy the a priori assumptions in Section \ref{sec: coefficient}. Let $\CC_1, \CC_2$ be the corresponding local Cauchy data and assume that $d(\CC_1,\CC_2)<1$. Then there exists a constant $C>0$ depending on the a-priori data only such that
    \begin{equation}
       \|\sigma^{(1)}-\sigma^{(2)}\|_{L^{\infty}(\Omega)}+\|q^{(1)}-q^{(2)}\|_{L^{\infty}(\Omega)} \leq C d(\CC_1, \CC_2).
    \end{equation}
\end{theorem}
The following Corollary is a straightforward consequence of the proof of Theorem \ref{stability}, hence we state the result and omit the proof.
\begin{corollary}\label{corollary}
    Under the assumptions of Theorem \ref{stability}, there exist constants $C>0$, $0<\eta<1$ depending on the a-priori data only such that
    \begin{equation}
        \|\sigma^{(1)}-\sigma^{(2)}\|_{L^{\infty}(\Sigma)} + \|q^{(1)}-q^{(2)}\|_{L^{\infty}(\Sigma)} \leq C(d(\CC_1,\CC_2) + E)^{1-\eta} d(\CC_1,\CC_2)^{\eta},
    \end{equation}
    with $E=\max\{\|\sigma^{(1)}-\sigma^{(2)}\|_{L^{\infty}(\Sigma)}, \|q^{(1)}-q^{(2)}\|_{L^{\infty}(\Sigma)}\}$.
\end{corollary}

\section{Auxiliary Propositions}\label{sec: 3}

The proof of Theorem \ref{stability} is based on the method of singular solutions and the quantitative estimates of unique continuation. In this section we introduce the main tools and propositions needed. In Section \ref{subsec: green} we define the Green's functions and we describe their asymptotic behaviour near the discontinuity interfaces. The Green functions are weak solutions to a well-posed boundary value problem defined on an enlarged domain $\Omega_0$ with Robin data prescribed on a small portion of $\p\Omega_0$ that is not contained in $\p\Omega$. 
In Section \ref{sec: qeuc} we introduce the singular integrals and the quantitative estimates of propagation of smallness.

\subsection{Green functions and asymptotic estimates}\label{subsec: green}
We introduce the enlarged domain
\begin{equation}
    \Omega_0=\overset{\circ}{\overline{\Omega\cup D_0}},
\end{equation}
where $D_0\subset (\R^n\setminus \Omega)$ is a measurable domain with boundary of class $C^2$ such that $\p D_0 \cap \p\Omega \subseteq \Sigma$. Let $\Sigma_0\subset (\p\Omega_0\setminus \p\Omega)$ be a non-empty flat portion of size $r_0\slash 3$. For $r>0$, consider the domain perturbation
$(\Omega_0)_r = \left\{x\in\Omega_0\,:\,\text{dist}(x,\p\Omega_0)>r \right\}.$

Let $\{\sigma,q\}$ be a pair of coefficients that satisfy the assumptions of Section \ref{sec: coefficient}. We extend them on $D_0$ by setting $\sigma|_{D_0}=Id_n$, $\gamma|_{D_0}=1$ and $q|_{D_0}=1$, where $Id_n$ denotes the $n\times n$ identity matrix. With an abuse of notation, we denote with the same letters the two extended coefficients when we deal with the enlarged domain $\Omega_0$.

Let $G$ be the Green function associated to the elliptic operator $\text{div}(\sigma \nabla \cdot) + \cdot$. For any $y\in\Omega_0$, let $G(\cdot,y)$ be the unique distributional solution to the problem
\begin{equation}\label{eqn: green system}
    \begin{cases}
        \,\mbox{div}(\sigma \nabla G(\cdot,y)) + q G(\cdot,y) =0, &\mbox{in }\Omega_0,\\
        \,G(\cdot,y) = 0, &\mbox{on }\p\Omega_0\setminus \Sigma_0,\\
        \,\sigma \nabla G(\cdot,y)\cdot \nu + i G(\cdot,y) = 0, &\mbox{on }\Sigma_0,
    \end{cases}
\end{equation}
(see \cite[Lemma 4.1]{Foschiatti2023}).
Moreover, there exists a positive constant $C$ that depends on $\lambda, n$ only such that
\begin{equation}\label{eqn: green upper bound}
    0<|G(x,y)|<C|x-y|^{2-n}, \qquad \text{for any }x,y\in \Omega_0, \, x\neq y.
\end{equation}
Fix an index $m\in \{0,\dots,N-1\}$, let $P_{m+1}\in\Sigma_{m+1}$ and assume that, up to a rigid transformation, $P_{m+1}$ coincides with the origin $0$ and $\Sigma_{m+1}$ is a flat hypersurface of size $r_0\slash 3$. Define the following quantities:
\[
\gamma^+=\gamma_{m+1}(0), \qquad \gamma^-=\gamma_m(0), \qquad J=\sqrt{A(0)^{-1}}, \qquad |J|=\mbox{det}J.
\]
The fundamental solution $H$ associated to the elliptic operator div$((\gamma^-\chi_{\R^n_-}(\cdot) + \gamma^+ \chi_{\R^n_+}(\cdot)) A(0) \nabla\cdot)$ in $\R^n$ is given by the formula
\begin{equation}\label{eqn: fundamentalsol}
H(x,y)= |J|\displaystyle
    \begin{cases}
        \frac{1}{\gamma^+} \Gamma(Lx,Ly) + \frac{\gamma^+-\gamma^-}{\gamma^+(\gamma^++\gamma^-)} \Gamma(Lx,(Ly)^*) &\mbox{if }x_n, (Ly)_n>0,\\
        \frac{2}{\gamma^+ + \gamma^-} \Gamma(Lx,Ly) &\mbox{if }x_n\cdot (Ly)_n<0,\\
        \frac{1}{\gamma^-} \Gamma(Lx,Ly) + \frac{\gamma^--\gamma^+}{\gamma^-(\gamma^+ + \gamma^-)} \Gamma(Lx,(Ly)^*) &\mbox{if }x_n, (Ly)_n<0,
    \end{cases}
\end{equation}
where $L:\R^n\rightarrow \R^n$ is a linear map such that $L^{-1}\cdot (L^{-1})^T = A(0)$ (see \cite{Foschiatti2021} or \cite{Gaburro2015}).

\begin{proposition}\label{prop: asymptotic estimates}
    Fix $m\in \{0,\dots,N-1\}$. Let $Q_{m+1}\in B_{\frac{r_0}{4}}(P_{m+1})\cap \Sigma_{m+1}$, where $\Sigma_{m+1}$ is the flat portion defined in Section \ref{sec: domain}.
    For $r\in (0,\frac{r_0}{8})$, set $y_{m+1}=Q_{m+1} - r\nu(Q_{m+1})$, where $\nu(Q_{m+1})$ is the outward unit normal of $\p D_{m}$ at $P_{m+1}$ and let $x\in B_{\frac{r_0}{4}}(Q_{m+1})\cap D_{m+1}$. Then there exist $C_1, C_2, C_3$ positive constants, $0<\theta_1,\theta_2,\theta_3<1$ that depend on the a priori data only such that
    \begin{eqnarray}
        |\nabla_x G(x,y_{m+1}) - \nabla_x H(x,y_{m+1})| &\leq& C_1 |x-y_{m+1}|^{1-n+\theta_1},\label{eqn: asymptotic1}\\
        |\nabla_x \nabla_y G(x,y_{m+1}) - \nabla_x \nabla_y H(x,y_{m+1})| &\leq& C_3 |x-y_{m+1}|^{-n+\theta_2},\label{eqn: asymptotic2}\\
        |\nabla_y G(x,y_{m+1}) - \nabla_y H(x,y_{m+1})| &\leq& C_2 |x-y_{m+1}|^{1-n+\theta_3},\label{eqn: asymptotic}\\
        |\nabla^2_y G(x,y_{m+1}) - \nabla^2_y H(x,y_{m+1})| &\leq& C_2 |x-y_{m+1}|^{1-n}.\label{eqn: asymptotic3}
    \end{eqnarray}
\end{proposition}
For a proof of \eqref{eqn: asymptotic1} and \eqref{eqn: asymptotic2}, see \cite[Proposition 4.3]{Foschiatti2023}, \cite[Proposition 3.1]{Foschiatti2021} and \cite[Proposition 3.4]{Alessandrini2018}. For \eqref{eqn: asymptotic} and \eqref{eqn: asymptotic3}, see Section \ref{sec5}.

\subsection{Quantitative estimates of unique continuation}\label{sec: qeuc}
Consider the following sets:
\[
\mathcal{U}_0 = \Omega, \qquad \mathcal{W}_k=\overset{\circ}{\Big(\bigcup_{m=0}^k \overline{D_m}\Big)}, \qquad \mathcal{U}_k = \Omega\setminus \overline{\mathcal{W}_k},\quad \mbox{for }k=1,\dots,N.
\]
For $y,z\in \mathcal{W}_k$, define the singular solution
\begin{align}\label{def: singular1}
\begin{split}
    S_k(y,z) = \int_{\mathcal{U}_k} &(\sigma^{(1)}-\sigma^{(2)})(x) \nabla_x G_1(x,y)\cdot \nabla_x G_2(x,z) \diff x + \\
    &+\int_{\mathcal{U}_k} (q^{(2)}-q^{(1)})(x) G_1(x,y) G_2(x,z) \diff x,
    \end{split}
\end{align}
where $G_j$ are the weak solutions to \eqref{eqn: green system}.
Moreover, the following partial derivatives are well defined:
\begin{align}\label{def: singular2}
    \p_{y_i} \p_{z_j} S_k(y,z) = \int_{\mathcal{U}_k} &(\sigma^{(1)}-\sigma^{(2)})(x) \p_{y_i}\nabla_x G_1(x,y)\cdot \p_{z_j}\nabla_x G_2(x,z) \diff x +\notag \\
    &+ \int_{\mathcal{U}_k} (q^{(2)}-q^{(1)})(x) \p_{y_i}G_1(x,y) \p_{z_j}G_2(x,z) \diff x,
\end{align}
and
\begin{align}
    \p^2_{y_i y_j} \p^2_{z_i z_j} S_k(y,z) = \int_{\mathcal{U}_k} &(\sigma^{(1)}-\sigma^{(2)})(x) \p^2_{y_i y_j}\nabla_x G_1(x,y)\cdot \p^2_{z_i z_j}\nabla_x G_2(x,z) \diff x +\notag \\
    &+ \int_{\mathcal{U}_k} (q^{(2)}-q^{(1)})(x) \p^2_{y_i y_j}G_1(x,y) \p^2_{z_i z_j}G_2(x,z) \diff x,
\end{align}
for $i,j=1,\dots,n$. For any $y,z\in \mathcal{W}_k$, one can show that $S_k(\cdot,z), S_k(y,\cdot) \in H^1_{loc}(\mathcal{W}_k)$ and are weak solutions, respectively, to
\begin{align*}
    \mbox{div}_y(\sigma^{(1)}\nabla_y S_k(\cdot,z)) + q^{(1)} S_k(\cdot,z) = 0 \qquad \mbox{in }\mathcal{W}_k,\\
    \mbox{div}_z(\sigma^{(2)}\nabla_z S_k(y,\cdot)) + q^{(2)} S_k(y,\cdot) = 0 \qquad \mbox{in }\mathcal{W}_k,
\end{align*}
(see \cite[Proposition 3.3]{Alessandrini2005}). Set
\begin{equation*}
E := \max\{\|\gamma^{(1)}-\gamma^{(2)}\|_{L^{\infty}(\Omega)}, \|q^{(1)}-q^{(2)}\|_{L^{\infty}(\Omega)}\}.
\end{equation*}
Notice that for any $y,z\in\mathcal{W}_k$,
\[
|S_k(y,z)|\leq C E (d(y)d(z))^{1-\frac{n}{2}},
\]
where $d(y)$ is the distance of $y$ from $\mathcal{U}_k$ and $C$ is a positive constant that depend on $\bar{A}$ and the a priori data.

The following Proposition introduces the quantitative estimates of unique continuation for the singular integrals.

\begin{proposition}\label{proposizione unique continuation finale}
	Suppose that for some positive $\varepsilon_0$ we have
	\begin{equation}\label{estim0}
		\left|S_k(y,z)\right|\leq
		r_0^{2-n}\varepsilon_0,\quad \mbox{for every}\: (y,z)\in
		D_0\times D_0.
	\end{equation}
Then there exist $\bar{r}>0$, $C>0$ constants that depend on the a priori data only such that the following inequalities hold true for every $r\in (0,\bar{r}\slash 8)$:
\begin{eqnarray}
    \left|S_k\left(y_{k+1},y_{k+1})\right)\right|&\leq&C r^{-2\tilde\gamma}\left( \frac{\ep_0}{\ep_0+E} \right)^{\tau_r^2 \beta^{2N_1}} (\ep_0 +E),\label{estim1}\\
    \left|\partial_{y_j}\partial_{z_i}S_{k}\left(y_{k+1},y_{k+1}\right)\right| &\leq& C r^{-2\tilde\gamma-2}\left( \frac{\ep_0}{\ep_0+E} \right)^{\tau_r^2 \beta^{2N_1}} (\ep_0 +E),\label{estim2}\\
    |\p^2_{y_iy_j} \p^2_{z_i z_j} S_k(y_{k+1},y_{k+1})| &\leq& C r^{-2\tilde\gamma-4}\left( \frac{\ep_0}{\ep_0 +E} \right)^{\tau_r^2 \beta^{2N_1}} (\ep_0 +E),\label{estim3}
 \end{eqnarray}
	for any $i,j=1,\dots , n$,
	$y_{k+1}=P_{k+1}- r\nu(P_{k+1})$, 
	$\nu(P_{k+1})$ is the exterior unit normal to $\partial{D}_{k}$  at the point $P_{k+1}$, $\tilde \gamma = \frac{n}{2}-1$, $0<\beta<1$, $N_1\in\N$ and
    \[
    \tau_r = \ln\left(\frac{12r_1-2r}{12r_1-3r}\right)\slash \ln\left(\frac{6r_1-r}{2r_1} \right).
    \]
\end{proposition}
\begin{remark}
    Notice that since
\begin{equation}\label{eqn: taurestimate}
\frac{\tau_r}{r} \geq \frac{1}{12 r_1 \ln 3},
\end{equation}
one can replace $\tau_r$ with $r$ in Proposition \ref{proposizione unique continuation finale}.
\end{remark}
For any $\eta>0$, let $\omega_{\eta}(t)$ be the non-decreasing function defined on $[0,+\infty)$ as
\begin{equation}
\omega_{\eta}(t)=
    \begin{cases}
        2^{\eta} e^{-2} |\ln t|^{-\eta}, &t\in (0,e^{-2}),\\
        e^{-2}, &t\in [e^{-2},+\infty).
    \end{cases}
\end{equation}
Recall that  
\[
[0,+\infty)\ni t \rightarrow t\omega_{\eta}\left(\frac{1}{t}\right)\in [0,+\infty)\quad \text{is a non-decreasing function}
\]
and for any $\beta\in(0,1)$,
\[
\omega_{\eta}\left(\frac{t}{\beta}\right)\leq |\ln e\beta^{-\frac{1}{2}}|^{\eta} \omega_{\eta}(t), \qquad \omega_{\eta}(t^{\beta}) \leq \left(\frac{1}{\beta} \right)^{\eta} \omega_{\eta}(t). 
\]
We set $\omega^{(0)}_{\eta}=t^{\eta}$ for $0<\eta<1$ and we denote the iterated composition of $\omega$ with itself as
\[
\omega^{(1)}_{\eta} = \omega_{\eta}, \quad \omega^{(j)}_{\eta} = \omega_{\eta}\circ \omega^{(j-1)}_{\eta}\quad \text{for }j=2,3,\dots
\]

Before proving Theorem \ref{stability}, we recall some useful formulas. Let $u_i\in H^{1}(\Omega)$ for $i=1,2$ be two weak solutions to
\[
\mbox{div}(\sigma^{(i)} \nabla u_i) + q^{(i)} u_i = 0\quad \mbox{in }\Omega,
\]
with $u_i|_{\p\Omega}\in H^{\frac{1}{2}}_{00}(\Sigma)$. By the weak formulation for $i=1,2$, one derives
\begin{align}\label{eqn: 1}
    \int_{\Omega} [(&\sigma^{(1)}-\sigma^{(2)})(x) \nabla u_1\cdot \nabla u_2 + (q^{(2)}-q^{(1)})(x) u_1(x) u_2(x)] \diff x =\notag\\
    &= \int_{\Sigma} [\sigma^{(2)}\nabla \bar{u}_2\cdot \nu\, \bar{u}_1 - \sigma^{(1)} \nabla u_1\cdot\nu\, u_2 ] \diff S(x).
\end{align}
For $v_i\in H^1(\Omega)$ a different solution to div$(\sigma^{(i)}\nabla v_i) + q^{(i)} v_i=0$ in $\Omega$ with $v_i\in H^{\frac{1}{2}}_{00}(\Sigma)$, one derives the following identity:
\begin{equation}\label{eqn: 2}
    \int_{\Sigma} [\sigma^{(i)}\nabla v_i\cdot\nu\, u_i - \sigma^{(i)}\nabla \bar{u}_i\cdot\nu\, \bar{v}_i] \diff S(x) = 0.
\end{equation}
By setting $i=2$ in \eqref{eqn: 2} and by summing up \eqref{eqn: 1} with \eqref{eqn: 2} one derives
\begin{align}\label{eqn: equality Cauchy}
    \int_{\Omega} &[(\sigma^{(1)}-\sigma^{(2)})(x)\nabla u_1(x)\cdot\nabla u_2(x) + (q^{(2)}-q^{(1)})(x) u_1(x) u_2(x)] \diff x =\notag\\
    &= \la \sigma^{(2)}\nabla \bar{u}_2\cdot \nu, (u_1-v_2)\ra - \la \sigma^{(1)}\nabla u_1\cdot\nu - \sigma^{(2)}\nabla v_2\cdot \nu, \bar{u}_2\ra.
\end{align}
If one takes the modulus in \eqref{eqn: equality Cauchy}, one obtains the following inequality:
\begin{align}\label{eqn: inequality Cauchy}
    \Big|\int_{\Omega} &[(\sigma^{(1)}-\sigma^{(2)})(x)\nabla u_1(x)\cdot\nabla u_2(x) + (q^{(2)}-q^{(1)})(x) u_1(x) u_2(x)] \diff x \Big|\leq\notag\\
    &\leq d(\CC_1,\CC_2)\,\|(u_1, \sigma^{(1)}\nabla u_1\cdot \nu)\|_{\HH}\, \|(\bar{u}_2, \sigma^{(2)}\nabla \bar{u}_2\cdot \nu)\|_{\HH}.
\end{align}

\section{Proof of Theorem \ref{stability}}\label{sec: 4}

\begin{proof}[Proof of Theorem \ref{stability}]
Let $\{\sigma^{(i)},q^{(i)}\}$ for $i=1,2$ be two pairs of coefficients that satisfy the assumptions of Section \ref{sec: coefficient} and let $\mathcal{C}_1, \mathcal{C}_2$ be the corresponding local Cauchy data. Due to the nature of the leading order coefficient, by \eqref{Abar}, the following inequality 
\[
\|\sigma^{(1)}-\sigma^{(2)}\|_{L^{\infty}(\Omega)} \leq C d(\mathcal{C}_1, \mathcal{C}_2)
\]
is equivalent to
\[
\|\gamma^{(1)}-\gamma^{(2)}\|_{L^{\infty}(\Omega)} \leq C d(\mathcal{C}_1,\mathcal{C}_2),
\]
where $C>1$ is a constant that depends on the a priori data only. 

For $K\in \{1,\dots,N\}$, let $D_K$ be the subdomain of the known partition of $\Omega$ such that
\begin{align*}
\|\gamma^{(1)}-\gamma^{(2)}\|_{L^{\infty}(\Omega)} 
=\|\gamma_K^{(1)}-\gamma_K^{(2)}\|_{L^{\infty}(D_K)}.
\end{align*}
Similarly, for $\tilde K\in \{1,\dots,N\}$, let $D_{\KK}$ be such that
\begin{align*}
\|q^{(1)}-q^{(2)}\|_{L^{\infty}(\Omega)} 
=\|q_{\KK}^{(1)}-q_{\KK}^{(2)}\|_{L^{\infty}(D_{\KK})}.
\end{align*}
Our goal is to prove that
\[
\|q_{\KK}^{(1)}-q_{\KK}^{(2)}\|_{L^{\infty}(D_{\KK})}+\|\gamma_K^{(1)}-\gamma_K^{(2)}\|_{L^{\infty}(D_K)} \leq C d(\mathcal{C}_1,\mathcal{C}_2).
\]
Let $\Omega_0$ be the augmented domain and let $\sigma^{(i)}$ and $q^{(i)}$ for $i=1,2$ be the extended coefficient on $D_0$, with $\sigma^{(i)}|_{D_0}=Id_n$ and $q^{(i)}=1$. Let $D_0,D_1,\dots ,D_K$ be the chain of contiguous domains such that $\Sigma_m=\p D_m\cap \p D_{m+1}$ and $\Sigma_1=\p D_0\cap \p D_1$. Set
\begin{align*}
\ep &=d(\mathcal{C}_1,\mathcal{C}_2),\\
E &=\max\{\|\gamma_K^{(1)}-\gamma_K^{(2)}\|_{L^{\infty}(D_K)}, \|q_{\tilde K}^{(1)}-q_{\tilde K}^{(2)}\|_{L^{\infty}(D_{\tilde K})}\},\\
\delta_k &=\|\gamma^{(1)}-\gamma^{(2)}\|_{L^{\infty}(\mathcal{W}_k)},\\
\tilde \delta_k &=\|q^{(1)}-q^{(2)}\|_{L^{\infty}(\mathcal{W}_k)},\\
\delta^*_k &= \max\{\delta_k, \tilde\delta_k\}\qquad \mbox{for }k=1,\dots, \max\{K,\tilde K\}.
\end{align*}

Let $\{x_1,\dots,x_n\}$ be a coordinate system with origin at $P_k$. Let $\Sigma_k$ be the flat interface of Section \ref{sec: domain}. We assume that it is contained in the tangential hypersurface of $\p D_1 \cap B_{\frac{r_0}{4}}$ at $P_k$. For any scalar function $f$, we denote with $D_T f(x)$ the tangential field of $f$ at $x$ and with $\p_{\nu} f(x)$ the derivative in the normal direction at $x$. The affine function $(\gamma^{(1)}_k-\gamma^{(2)}_k)$ can be bounded from above in $D_k$ in terms of the quantities
\begin{equation}
\|\gamma^{(1)}_k - \gamma^{(2)}_k\|_{L^{\infty}(\Sigma_k\cap B_{\frac{r_0}{4}}(P_k))}\quad\mbox{and}\quad | \partial_{\nu}(\gamma^{(1)}_k - \gamma^{(2)}_k)(P_k)|.
\end{equation}
Indeed, set
\[
A_k + B_k\cdot x = (\gamma^{(1)}_k -\gamma^{(2)}_k)(x), \qquad A_k\in \R,\,B_k\in \R^n,\,x\in D_k.
\] 
Fix an orthonormal basis $\{e_j\}_{j=1}^{n-1}$ of $\Sigma_k$ and let $e_n$ be the direction of the normal. One can evaluate $(\gamma^{(1)}_k - \gamma^{(2)}_k)$ at the points $P_k$ and $P_k + \frac{r_0}{6}e_j$ for $j=1,\dots,n$ and derive
\[
|A_k + B_k\cdot P_k| + \frac{r_0}{6} \sum_{j=1}^{n-1} |(B_k)_j| \leq C \|\gamma^{(1)}_k - \gamma^{(2)}_k\|_{L^{\infty}(\Sigma_k\cap B_{\frac{r_0}{4}}(P_k))}
\]
and
\[
|B_k\cdot e_n| = |\p_{\nu}(\gamma^{(1)}_k-\gamma^{(2)}_k)(P_k)|.
\]
Hence, it turns out that
\[
\|\gamma^{(1)}_k - \gamma^{(2)}_k\|_{L^{\infty}(D_k)} \leq C\left( \|\gamma^{(1)}_k - \gamma^{(2)}_k\|_{L^{\infty}(\Sigma_k\cap B_{\frac{r_0}{4}}(P_k))} + |\p_{\nu}(\gamma^{(1)}_k -\gamma^{(2)}_k)(P_k)|\right),
\]
for $C>0$ constant that depends on the a priori data. 
\bigskip

Our goal is to estimate $\delta_k^*$ for any $k=1,\dots, \max\{K,\tilde K\}$.

When $k=1,$ we obtain the following H\"older estimates at the boundary:
\begin{align}
    \delta_1 &\leq C (E+\ep) \Big(\frac{\ep}{\ep + E} \Big)^{\eta_1}\\
    \tilde \delta_1 &\leq C (E+\ep) \Big(\frac{\ep}{\ep + E} \Big)^{\tilde \eta_1},
\end{align}
with $0<\eta_1,\tilde \eta_1 <1$ that depend on $\theta_1, \theta_2, \theta_3$ and $C$ are positive  constants that depend on the a priori data only (see the Appendix for a proof).

We proceed by estimating $\delta^*_2$. We claim that the following inequalities hold:
\begin{align}
    \delta_2 &\leq C (\ep + E) \omega^{(2)}_{\eta_2}\Big(\frac{\ep}{\ep+E} \Big),\\
    \tilde \delta_2  &\leq C (\ep + E) \omega^{(3)}_{\tilde \eta_2}\Big(\frac{\ep}{\ep+E} \Big),
\end{align}
with $0<\eta_2, \tilde \eta_2<1$.

Our idea is to first estimate the $L^2$ norm of $(\gamma^{(1)}-\gamma^{(2)})$ on $\mathcal{W}_1$, namely $\delta_2$, by means of $\delta_1^*$ and then to estimate the $L^2$ norm of $(q^{(2)}-q^{(1)})$ on $\mathcal{W}_1$, namely $\tilde \delta_2$, in terms of $\delta_2$.
\bigskip

 For any $y,z\in D_0$, the following identities hold:
\begin{equation}\label{eqn: green identity1}
    \begin{split}
        &\int_{\Sigma} [\sigma^{(2)}(x)\nabla_x G_2(x,z)\cdot \nu\, G_1(x,y) - \sigma^{(1)}(x) \nabla_x G_1(x,y)\cdot\nu\, G_2(x,z) ] \diff S(x) =\\
        &=S_1(y,z) + \int_{\mathcal{W}_1} [(\sigma^{(1)}-\sigma^{(2)})(x) \nabla_x G_1(x,y)\cdot \nabla_x G_2(x,z) + (q^{(2)}-q^{(1)})(x) G_1(x,y) G_2(x,z)] \diff x,
    \end{split}
\end{equation}
and
\begin{equation}\label{eqn: 1der}
    \begin{split}
        \int_{\Sigma} [&\sigma^{(2)}(x)\nabla_x \p_{z_n} G_2(x,z)\cdot \nu\, \p_{y_n} G_1(x,y) - \sigma^{(1)}(x) \nabla_x \p_{y_n} G_1(x,y)\cdot\nu\, \p_{z_n} G_2(x,z) ] \diff S(x) =\\
        &=\p_{y_n} \p_{z_n} S_1(y,z) + \int_{\mathcal{W}_1} (\sigma^{(1)}-\sigma^{(2)})(x) \nabla_x \p_{y_n} G_1(x,y)\cdot \nabla_x \p_{z_n} G_2(x,z) \diff x + \\
        &+\int_{\mathcal{W}_1} (q^{(2)}-q^{(1)})(x) \p_{y_n} G_1(x,y) \p_{z_n} G_2(x,z) \diff x.
    \end{split}
\end{equation}
By \eqref{eqn: inequality Cauchy}, 
 \begin{align}\label{eqn: bound S0}
 \begin{split}
\Big|\int_{\Sigma} [&\sigma^{(2)}\nabla_x G_2(x,z)\cdot \nu\, G_1(x,y) - \sigma^{(1)} \nabla_x G_1(x,y)\cdot\nu\, G_2(x,z) ] \diff S(x)\Big| \leq\\
&\leq C \ep\, (d(y) d(z))^{1-\frac{n}{2}},
\end{split}
\end{align}
where $d(y)$ denotes the distance between $y$ and $\Omega$. 
Let $\rho=r_0\slash 4$, let $r\in(0,\bar{r}\slash 8)$ and set $w=P_2 + r\nu(P_2)$, where $\nu(P_2)$ is the outward unit normal of $\p D_2$ at $P_2$. Consider
\begin{equation}
    S_1(w,w) = I_1(w)+I_2(w),
\end{equation}
with
\begin{equation*}
    \begin{split}
        I_1(w) &= \int_{B_{\rho}(P_2)\cap D_2}(\gamma_2^{(1)}-\gamma_2^{(2)})(x)\:A(x)\,\nabla_x G_1(x,w)\cdot\nabla_x G_2(x,w) \diff x +\\
         &+\int_{B_{\rho}(P_2)\cap D_2}(q_2^{(2)}-q_2^{(1)})(x)\, G_1(x,w)\cdot G_2(x,w) \diff x,
    \end{split}
\end{equation*}
and
\begin{equation*}
    \begin{split}
        I_2(w) &= \int_{\mathcal{U}_2 \setminus (B_{\rho}(P_2)\cap D_2)}(\sigma^{(1)}-\sigma^{(2)})(x)\,\nabla_x G_1(x,w)\cdot\nabla_x G_2(x,w) \diff x +\\
         &+\int_{\mathcal{U}_2 \setminus (B_{\rho}(P_2)\cap D_2)}(q^{(2)}-q^{(1)})(x)\, G_1(x,w)\cdot G_2(x,w) \diff x.
    \end{split}
\end{equation*}
The volume integrals of $I_2(w)$ can be bounded from above via Caccioppoli inequality (see also \cite[Proposition 3.1]{Alessandrini2005}):
\begin{equation}\label{11 stima I2}
    |I_2(w)|\leq C E \rho^{2-n}.
\end{equation}
	
Regarding $I_1(w)$, notice that there exists $x^*\in \overline{\Sigma_2\cap B_{r_0\slash 4}(P_2)}$ such that
\begin{equation}\label{expansion1}
    (\gamma^{(1)}_2 - \gamma^{(2)}_2)(x^*) = \|\gamma^{(1)}_2 -\gamma^{(2)}_2\|_{L^{\infty}(\Sigma_2 \cap B_{r_0\slash 4}(P_2))}.
\end{equation}
By \eqref{expansion1},
\begin{equation*}
    \begin{split}
        I_1(w) &= \int_{B_{\rho}(P_2)\cap D_2}(\gamma_2^{(1)}-\gamma_2^{(2)})(x^*)\:A(x)\,\nabla_x G_1(x,w)\cdot\nabla_x G_2(x,w) \diff x +\\
        &+ \int_{B_{\rho}(P_2)\cap D_2} B_2 \cdot(x-x^*)\:A(x)\,\nabla_x G_1(x,w)\cdot\nabla_x G_2(x,w) \diff x +\\
         &+\int_{B_{\rho}(P_2)\cap D_2}(q_2^{(2)}-q_2^{(1)})(x)\, G_1(x,w)\cdot G_2(x,w) \diff x.
    \end{split}
\end{equation*}
By the asymptotic estimate \eqref{eqn: asymptotic1}, one obtains
\begin{align*}
    I_1(w) \geq &\|\gamma^{(1)}_2 - \gamma^{(2)}_2\|_{L^{\infty}(\Sigma_2\cap B_{\frac{r_0}{4}}(P_2))} \Big\{ \int_{B_{\rho}(P_2)\cap D_2} A(x) \nabla_x H_1(x,w)\cdot \nabla_x H_2(x,w) \diff x -\\
    &-\int_{B_{\rho}(P_2)\cap D_2} |x-w|^{2(1-n)+\theta_1} \diff x - \int_{B_{\rho}(P_2)\cap D_2} |x-w|^{2(1-n +\theta_1)} \diff x\Big\} -\\
    &- C E \int_{B_{\rho}(P_2)\cap D_2} |x| |x-w|^{2(1-n)}\, \diff x - C E \int_{B_{\rho}(P_2)\cap D_2} |x-w|^{2(2-n)} \diff x.
\end{align*}
It turns out that
\begin{equation}\label{eqn: stabinterm11}
    |I_1(w)| \geq C \|\gamma^{(1)}_2 -\gamma^{(2)}_2\|_{L^{\infty}(\Sigma_2\cap B_{\frac{r_0}{4}}(P_2))} r^{2-n} - C E r^{2-n+\theta_1} - C E r^{3-n}.
\end{equation}
Notice that for any $y,z\in (D_0)_{r_0\slash 3},$
\[
|S_1(y,z)|\leq C r_0^{2-n} (\ep + \delta_1^*),
\]
hence, by \eqref{estim1}, 
\begin{equation}\label{eqn: intermstabs1}
    |S_1(y,z)| \leq C (\ep + \delta_1^* + E) \Big(\frac{\ep + \delta_1^*}{\ep + \delta_1^* + E} \Big)^{\beta^{2N_1}\tau_r^2} r^{2-n}.
\end{equation}
Since
\[
|I_1(w)|\leq |S_1(w,w)| + |I_2(w)|,
\]
if we rearrange the inequalities \eqref{eqn: stabinterm11} and \eqref{11 stima I2} together with \eqref{eqn: intermstabs1} and \eqref{eqn: taurestimate}, we derive 
\begin{equation}\label{eqn: intermstab21}
    \begin{split}
    \|\gamma^{(1)}_2 -\gamma^{(2)}_2&\|_{L^{\infty}(\Sigma_2\cap B_{\frac{r_0}{4}}(P_2))} r^{2-n} \le C (\ep + \delta^*_1 + E) \Big\{\Big(\frac{\ep + \delta^*_1}{\ep + \delta^*_1 + E} \Big)^{\beta^{2N_1}(12r_1\ln3)^{-2}r^2} r^{2-n} + r^{\theta_1} \Big\}.
    \end{split}
\end{equation}
If we multiply \eqref{eqn: intermstab21}, and if we choose
\[
r = \Big|\ln \Big(\frac{\ep + \delta^*_1}{\ep + \delta^*_1 + E} \Big) \Big|^{-\frac{1}{2+\theta_1}}
\]
(see also \cite{Beretta2013}), it turns out that
\begin{equation}
    \|\gamma^{(1)}_2-\gamma^{(2)}_2\|_{L^{\infty}(\Sigma_2\cap B_{\frac{r_0}{4}}(P_2))} \leq C (\ep+\delta_1^* + E) \Big|\ln \Big(\frac{\ep + \delta^*_1}{\ep + \delta^*_1 + E} \Big) \Big|^{-\frac{\theta_1}{2+\theta_1}}.
\end{equation}
By the properties of $\omega_b$, one derives
\begin{equation}\label{eqn: surfacel2}
    \|\gamma^{(1)}_2-\gamma^{(2)}_2\|_{L^{\infty}(\Sigma_2\cap B_{\frac{r_0}{4}}(P_2))} \leq C (\ep + E) \omega_{b} \Big( \frac{\ep}{\ep + E}\Big),
\end{equation}
with $0<b<1$ depending on $\theta_1$.

A similar estimate can be derived for $\p_{\nu}(\gamma^{(1)}_2-\gamma^{(2)}_2)$. From Taylor's formula, one derives
\begin{align*}
(\gamma^{(1)}_2-\gamma^{(2)}_2)(x) &= (\gamma^{(1)}_2-\gamma^{(2)}_2)(P_2) + (D_T(\gamma^{(1)}_2-\gamma^{(2)}_2)(P_2))\cdot (x-P_2)' +\\ &+(\p_{\nu}(\gamma^{(1)}_2 -\gamma^{(2)}_2)(P_2))\cdot (x-P_2)_n.
\end{align*}
Hence,
\begin{align*}
    |\p_{y_n} \p_{z_n} S_1(w,w)| &\geq  \Big|\int_{B_{\rho}(P_2)\cap D_2} \p_{\nu}(\gamma^{(1)}_2 - \gamma^{(2)}_2)(P_2)\cdot (x-P_2)_n A(x) \nabla_x \p_{y_n} G_1(x,w)\cdot \nabla_x \p_{z_n} G_2(x,w) \diff x \Big|-\\
    &- \Big|\int_{B_{\rho}(P_2)\cap D_2} D_T(\gamma^{(1)}_2 - \gamma^{(2)}_2)(P_2)\cdot (x-P_2)' A(x) \nabla_x \p_{y_n} G_1(x,w)\cdot \nabla_x \p_{z_n} G_2(x,w) \diff x \Big|-\\
    &- \Big|\int_{B_{\rho}(P_2)\cap D_2} (\gamma^{(1)}_2 - \gamma^{(2)}_2)(P_2) A(x) \nabla_x \p_{y_n} G_1(x,w)\cdot \nabla_x \p_{z_n} G_2(x,w) \diff x \Big|-\\
    &- \Big|\int_{B_{\rho}(P_2)\cap D_2} (q^{(2)}_2 - q^{(1)}_2)(x) \p_{y_n} G_1(x,w)\cdot \p_{z_n} G_2(x,w) \diff x \Big|\\
    &- \Big|\int_{\mathcal{U}_2 \setminus (B_{\rho}(P_2)\cap D_2)} (\sigma^{(1)}-\sigma^{(2)})(x) \p_{y_n} \nabla_x G_1(x,w)\cdot \p_{z_n} \nabla_x G_2(x,w) \diff x \Big|-\\
    &- \Big|\int_{\mathcal{U}_2 \setminus (B_{\rho}(P_2)\cap D_2)} (q^{(1)}-q^{(2)})(x) \p_{y_n} G_1(x,w)\cdot \p_{z_n} G_2(x,w) \diff x\Big|\\
    &= I_{11} - I_{12} - I_{13} - I_{14} - I_{15} - I_{16}.
\end{align*}
To estimate $I_{11}$ from below, we add and subtract the fundamental solution $H_i$, $i=1,2$, and by \eqref{eqn: asymptotic2}, one derives
\begin{equation}
    I_{11}\geq C |\p_{\nu}(\gamma^{(1)}_2 - \gamma^{(2)}_2)(P_2)| r^{1-n} - C E r^{1-n+\theta_2}.
\end{equation}
To estimate the terms $I_{12}$ and $I_{13}$, notice that by \eqref{eqn: surfacel2}
\[
|(\gamma^{(1)}_2 - \gamma^{(2)}_2)(P_2)| + C |D_T(\gamma^{(1)}_2 - \gamma^{(2)}_2)(P_2)| \leq C \|\gamma^{(1)}_2 - \gamma^{(2)}_2\|_{L^{\infty}(\Sigma_2 \cap B_{\frac{r_0}{4}})} \leq C (\ep + E) \omega_{b} \Big( \frac{\ep}{\ep + E}\Big).
\]
Regarding the integral $I_{14}$, one bounds it from above as
\begin{align*}
    I_{14} &\leq \|q^{(2)}_2-q^{(1)}_2\|_{L^{\infty}(D_2)} \int_{D_2\cap B_{\rho}} |\p_{y_n} G_1(x,w)| |\p_{z_n} G_2(x,w)| \diff x\\
    &\leq C \int_{D_2\cap B_{\rho}} |x-w|^{2(1-n)} \leq C\,r^{2-n}.
\end{align*}
The integral $I_{15}$ and $I_{16}$ can be bounded by means of \cite[Proposition 3.1]{Alessandrini2005} as
\[
I_{15}, I_{16} \leq C E \rho^{-n}.
\]
To sum up, we have 
\begin{equation}
|\p_{\nu}(\gamma^{(1)}_2 - \gamma^{(2)}_2)(P_2)| r^{1-n} \leq |\p_{y_n} \p_{z_n} S_1(w,w)| + C \{ Er^{1-n+\theta_2} + C (\ep + E) \omega_{b} \Big( \frac{\ep}{\ep + E}\Big) r^{-n} \}.
\end{equation}
Notice that for $y,z\in (D_0)_{r_0\slash 3}$, 
\[
|\p_{y_n}\p_{z_n} S_1(y,z)| \leq C (\ep + \delta_1^*) r^{-n},
\]
then by \eqref{estim2} and \eqref{eqn: taurestimate},
\begin{equation}\label{eqn: interms1der}
    \left|\partial_{y_j}\partial_{z_i}S_{1}(w,w)\right| \leq C (\ep+ \delta_1^* +E) \left( \frac{\ep+\delta_1^*}{\ep+ \delta_1^*+E} \right)^{\beta^{2N_1}(12r_1\ln3)^{-2}r^2} r^{-n}.
\end{equation}
Hence, one derives
\begin{equation}\label{eqn1: interm}
\begin{split}
|\p_{\nu}(\gamma^{(1)}_2 - \gamma^{(2)}_2)(P_1)| r^{1-n} \leq C \Big\{ Er^{1-n+\theta_2} +  &(\ep+ \delta_1^* +E) \left( \frac{\ep+\delta_1^*}{\ep+ \delta_1^*+E} \right)^{\beta^{2N_1}(12r_1\ln3)^{-2}r^2} r^{-n} + \\ &+ (\ep + E) \omega_{b} \Big( \frac{\ep}{\ep + E}\Big) r^{-n} \Big\}.
\end{split}
\end{equation}
Multiplying \eqref{eqn1: interm} by $r^{n-1}$ and optimizing w.r.t. $r$ leads to
\begin{equation}
    |\p_{\nu}(\gamma^{(1)}_1-\gamma^{(2)}_1)(P_1)| \leq C(\varepsilon + E) \omega^{(2)}_{\eta_2} \left(\frac{\varepsilon}{\varepsilon+ E}\right),
\end{equation}
with $0<\eta_2<1$. Hence, we conclude that
\begin{equation}\label{eqn: bound delta2}
    \delta_2 \leq C (\varepsilon + E) \omega^{(2)}_{\eta_2}\left(\frac{\varepsilon}{\varepsilon+E}\right).
\end{equation}
\bigskip

Our goal is to derive a bound for $\delta_2^*$. Notice that the norm ${\|q^{(2)}_2-q^{(1)}_2\|_{L^{\infty}(D_2)}}$ can be evaluated in terms of the following quantities:
\begin{equation}\label{eqn: upper q1 stab}
\|q^{(2)}_2-q^{(1)}_2\|_{L^{\infty}(\Sigma_2\cap B_{\frac{r_0}{4}}(P_2))} \quad \text{and}\quad |\p_{\nu}(q^{(2)}_2 - q^{(1)}_2)(P_2)|.
\end{equation}
 Let $\rho$, $r$, $w$ be as above. Consider
\[
\p_{y_n}\p_{z_n} S_1(w,w) = \p_{y_n}\p_{z_n} I_1(w) + \p_{y_n}\p_{z_n} I_2(w).
\]
We determine a lower bound for $\p_{y_n}\p_{z_n} I_1(w)$ in terms of $\|q^{(2)}_2-q^{(1)}_2\|_{L^{\infty}(\Sigma_2\cap B_{r_0\slash4}(P_2))}$. 
By the asymptotic estimate \eqref{eqn: asymptotic2} and \eqref{eqn: bound delta2}, one derives
\begin{align*}
    |\p_{y_n} \p_{z_n} &I_1(w)| \geq \|q^{(2)}_2-q^{(1)}_2\|_{L^{\infty}(\Sigma_2\cap B_{\frac{r_0}{4}}(P_2))}  \Big\{ \int_{B_{\rho}(P_2)\cap D_2} \p_{y_n} H_1(x,w)\,\p_{z_n} H_2(x,w) \diff x -\\
    &-\int_{B_{\rho}(P_2)\cap D_2} |x-w|^{2(1-n)+\theta_3} \diff x - \int_{B_{\rho}(P_2)\cap D_2} |x-w|^{2(1-n +\theta_3)} \diff x\Big\} -\\
    &- C E \int_{B_{\rho}(P_2)\cap D_2} |x| |x-w|^{2(1-n)}\, \diff x - C (\ep + E)\omega^{(2)}\Big(\frac{\ep}{\ep+E} \Big) \int_{B_{\rho}(P_2)\cap D_2} |x-w|^{-n} \diff x.
\end{align*}
It turns out that
\begin{align*}
C \|q^{(2)}_2-q^{(1)}_2\|_{L^{\infty}(\Sigma_2\cap B_{\frac{r_0}{4}}(P_2))} r^{2-n} \leq |\p_{y_n}\p_{z_n} I_1(w)| + C E r^{2-n+\theta_3} + C (\varepsilon + E) \omega^{(2)}_{\eta_2}\left(\frac{\varepsilon}{\varepsilon+E}\right) r^{-n}.
\end{align*}
By \eqref{eqn: interms1der}, due to the fact that
\[
|\p_{y_n}\p_{z_n} I_1(w)| \leq |\p_{y_n}\p_{z_n} S_1(w,w)| + |\p_{y_n}\p_{z_n} I_2(w)|,
\]
by the upper bound for $I_2(w)$, \eqref{estim2} and \eqref{eqn: taurestimate} we derive
\begin{align*}
\|q^{(2)}_2-q^{(1)}_2\|_{L^{\infty}(\Sigma_2\cap B_{\frac{r_0}{4}}(P_2))} r^{2-n} &\leq C\Big\{(\ep+ \delta_1^* +E) \left( \frac{\ep+\delta_1^*}{\ep+ \delta_1^*+E} \right)^{\beta^{2N_1}(12r_1\ln3)^{-2}r^2} r^{-n}+\\
&+ E r^{2-n+\theta_1} + (\varepsilon + E) \omega^{(2)}_{\eta_2}\left(\frac{\varepsilon}{\varepsilon+E}\right) r^{-n}\Big\}.
\end{align*}
Multiply by $r^{n-2}$ to obtain 
\begin{align*}
    \|q^{(2)}_2-q^{(1)}_2\|_{L^{\infty}(\Sigma_2\cap B_{\frac{r_0}{4}}(P_2))} &\leq C \Big\{(\ep+ \delta_1^* +E) \left( \frac{\ep+\delta_1^*}{\ep+ \delta_1^*+E} \right)^{\beta^{2N_1}(12r_1\ln3)^{-2}r^2} r^{-2}+\\
&+ E r^{\theta_1} + (\varepsilon + E) \omega^{(2)}_{\eta_2}\left(\frac{\varepsilon}{\varepsilon+E}\right) r^{-2}\Big\}.
\end{align*}
By optimizing with respect to $r$, one concludes that
\begin{equation}\label{eqn3: bound q1}
    \|q^{(2)}_2-q^{(1)}_2\|_{L^{\infty}(\Sigma_2\cap B_{\frac{r_0}{4}}(P_2))} \leq C (\varepsilon + E) \omega^{(3)}_{b_2}\left(\frac{\varepsilon}{\varepsilon+E}\right),
\end{equation}
with $0<b_2<1$ that depends on $\theta_1, \theta_2, \theta_3$.
\bigskip

To estimate $|\p_{\nu}(q^{(2)}_2 - q^{(1)}_2)(P_2)|$, consider the singular solution $\p^2_{y_i y_j}\p^2_{z_i z_j} S_1(w,w)$ and split it as the sum of the terms
\begin{align*}
I_1^{ij}(w) &= \int_{D_2\cap B_{\rho}(P_2)} (\sigma^{(2)}_2 - \sigma^{(1)}_2)(x) \nabla_x \p^2_{y_i y_j} G_1(x,w)\cdot \nabla_x \p^2_{z_i z_j} G_2(x,w) \diff x +\\
&+\int_{D_2\cap B_{\rho}(P_2)} (q^{(2)}_2 - q^{(1)}_2)(x) \p^2_{y_i y_j} G_1(x,w)\cdot \p^2_{z_i z_j} G_2(x,w) \diff x
\end{align*}
and
\begin{align*}
I_2^{ij}(w) &= \int_{\mathcal{U}_2 \setminus (D_2\cap B_{\rho}(P_2))} (\sigma^{(1)}-\sigma^{(2)})(x) \nabla_x \p^2_{y_i y_j} G_1(x,y_r)\cdot \nabla_x \p^2_{z_i z_j} G_2(x,y_r) \diff x +\\
&+ \int_{\mathcal{U}_2 \setminus (D_2\cap B_{\rho}(P_2))} (q^{(2)} - q^{(1)})(x) \p^2_{y_i y_j} G_1(x,w)\cdot \p^2_{z_i z_j} G_2(x,w) \diff x.
\end{align*}
Set $I_m(w)=\{I^{ij}_m(w)\}_{i,j=1,\dots,n}$ for $m=1,2$. Denote by $|I_m(w)|$ the Euclidean norm of the matrix $I_m(w)$. The upper bound for $|I_2(w)|$ is given by
\[
|I_2(w)| \leq C E \rho^{-(n+2)},
\]
where $C$ is a positive constant that depends on the a priori data only. For the lower bound for $I_1(w)$,  
\begin{align*}
    |I_1(w)| &\geq \frac{1}{n} \sum_{i,j=1}^n\Big\{\Big|\int_{D_2\cap B_{\rho}(P_2)} (\p_{\nu}(q^{(2)}_2 -q^{(1)}_2)(P_2))\cdot (x-P_2)_n \p^2_{y_i y_j} G_1(x,w)\cdot \p^2_{z_i z_j} G_2(x,w) \diff x\Big| - \\
    &- \Big|\int_{D_2\cap B_{\rho}(P_2)} (D_T(q^{(2)}_2 - q^{(1)}_2)(P_2))\cdot (x-P_2)' \p^2_{y_i y_j} G_1(x,w)\cdot \p^2_{z_i z_j} G_2(x,w) \diff x\Big|-\\
    &- \Big|\int_{D_2\cap B_{\rho}(P_2)} (q^{(2)}_2 - q^{(1)}_2)(P_2) \p^2_{y_i y_j} G_1(x,w)\cdot \p^2_{z_i z_j} G_2(x,w) \diff x\Big|\Big\}-\\
    &- \Big|\int_{D_2\cap B_{\rho}(P_2)} (\sigma^{(1)}-\sigma^{(2)})(x) \p^2_{y_i y_j}\nabla_x G_1(x,w)\cdot \p^2_{z_i z_j} \nabla_x G_2(x,w) \diff x \Big|.
\end{align*}
Since 
\[
|(q^{(2)}_2 - q^{(1)}_2)(P_2)|+C |(D_T(q^{(2)}_2-q^{(1)}_2)(P_2))| \leq C\|q^{(2)}_2-q^{(1)}_2\|_{L^{\infty}(\Sigma_2\cap B_{\frac{r_0}{4}}(P_2))},
\]
by \eqref{eqn3: bound q1} and \eqref{eqn: asymptotic3}, one derives
\begin{equation}\label{eqn3: interm 2det I1}
\begin{split}
|I_1(w)|\geq C |(\p_{\nu}(q^{(2)}_2-q^{(1)}_2)(P_2))| r^{1-n} - C (\varepsilon + E) \omega^{(3)}_{b_2}\left(\frac{\varepsilon}{\varepsilon+E}\right) r^{-n} -\\
- C E r^{1+\theta_2 -n} - C (E+\ep)\omega_{\eta_2}^{(2)}\Big(\frac{\ep}{\ep + E} \Big) r^{-2-n}.
\end{split}
\end{equation}
Since for $y,z\in (D_0)_{r_0\slash 3}$,
\begin{equation*}
    \begin{split}
        \int_{\Sigma} [&\sigma^{(2)}(x)\nabla_x \p^2_{z_n} G_2(x,z)\cdot \nu\, \p^2_{y_n} G_1(x,y) - \sigma^{(1)}(x) \nabla_x \p^2_{y_n} G_1(x,y)\cdot\nu\, \p^2_{z_n} G_2(x,z) ] \diff S(x) =\\
        &= \p^2_{y_n}\p^2_{z_n} S_1(y,z) + \int_{\mathcal{W}_1} (\sigma^{(1)}-\sigma^{(2)})(x) \nabla_x \p^2_{y_n} G_1(x,y)\cdot \nabla_x \p^2_{z_n} G_2(x,z)\diff x +\\
        &+ \int_{\mathcal{W}_1} (q^{(2)}-q^{(1)})(x) \p^2_{y_n} G_1(x,y) \p^2_{z_n} G_2(x,z) \diff x,
    \end{split}
\end{equation*}
by \eqref{estim3} and \eqref{eqn: taurestimate}, it turns out that
\begin{equation}\label{eqn3: interm 2derS0}
|\p^2_{y_n}\p^2_{z_n} S_1(y_r,y_r)| \leq C \left( \frac{\ep+\delta_1^* }{\ep+ \delta_1^* +E} \right)^{\beta^{2N_1}(12r_1\ln3)^{-2}r^2} (\ep+\delta_1^* +E) r^{-2-n}.
\end{equation}
Collecting together \eqref{eqn3: interm 2det I1} and \eqref{eqn3: interm 2derS0}, one derives
\begin{align*}
    |(\p_{\nu}(q^{(2)}_1-q^{(1)}_1)(P_1))| r^{1-n} &\leq C (\varepsilon + E) \omega^{(2)}_{\eta_2}\left(\frac{\varepsilon}{\varepsilon+E}\right) r^{-2-n} + C E r^{1+\theta_2 -n} +\\
    &+C \left( \frac{\ep+\delta_1^* }{\ep+ \delta_1^* +E} \right)^{\beta^{2N_1}(12r_1\ln3)^{-2}r^2} (\ep+\delta_1^* +E) r^{-2-n}.
\end{align*}
Multiply by $r^{n-1}$ the last equation and optimize with respect to $r$ leads to the estimate
\begin{equation}\label{eqn3: bound nuq1}
    |(\p_{\nu}(q^{(2)}_1-q^{(1)}_1)(P_1))| \leq C (E + \ep) \omega^{(3)}_{\tilde \eta_2} \left( \frac{\ep}{\ep + E}\right),
\end{equation}
with $0<\tilde \eta_2<1$ that depends on $\theta_1, \theta_2, \tilde b$.
\bigskip

For the general case, consider the following identities:
\begin{equation}
    \begin{split}
        &\int_{\Sigma} [\sigma^{(2)}(x)\nabla_x G_2(x,z)\cdot \nu\, G_1(x,y) - \sigma^{(1)}(x) \nabla_x G_1(x,y)\cdot\nu\, G_2(x,z) ] \diff S(x) =\\
        &=S_{k-1}(y,z) + \int_{\mathcal{W}_{k-1}} [(\sigma^{(1)}-\sigma^{(2)})(x) \nabla_x G_1(x,y)\cdot \nabla_x G_2(x,z) + (q^{(2)}-q^{(1)})(x) G_1(x,y) G_2(x,z)] \diff x,
    \end{split}
\end{equation}
by \eqref{eqn: 1der} one derives 
\begin{equation}
    \begin{split}
        &\int_{\Sigma} \sigma^{(2)}(x)\nabla_x \p_{z_n} G_2(x,z)\cdot \nu\, \p_{y_n} G_1(x,y) - \sigma^{(1)}(x) \nabla_x \p_{y_n} G_1(x,y)\cdot\nu\, \p_{z_n} G_2(x,z) ] \diff S(x) =\\
        &=\int_{\mathcal{W}_{k-1}} [(\sigma^{(1)}-\sigma^{(2)})(x) \nabla_x \p_{y_n} G_1(x,y)\cdot \nabla_x \p_{z_n} G_2(x,z)\diff x+\\
        &+ \int_{\mathcal{W}_{k-1}}(q^{(2)}-q^{(1)})(x) \p_{y_n} G_1(x,y) \p_{z_n} G_2(x,z) \diff x+ \p_{y_n} \p_{z_n} S_{k-1}(y,z),
    \end{split}
\end{equation}
and
\begin{equation}
    \begin{split}
        &\int_{\Sigma} [\sigma^{(2)}(x)\nabla_x \p^2_{z_n} G_2(x,z)\cdot \nu\, \p^2_{y_n} G_1(x,y) - \sigma^{(1)}(x) \nabla_x \p^2_{y_n} G_1(x,y)\cdot\nu\, \p^2_{z_n} G_2(x,z) ] \diff S(x) =\\
        &=\int_{\mathcal{W}_{k-1}} (\sigma^{(1)}-\sigma^{(2)})(x) \nabla_x \p^2_{y_n} G_1(x,y)\cdot \nabla_x \p^2_{z_n} G_2(x,z) \diff x +\\
        &+ \int_{\mathcal{W}_{k-1}} (q^{(2)}-q^{(1)})(x) \p^2_{y_n} G_1(x,y) \p^2_{z_n} G_2(x,z) \diff x+ \p^2_{y_n} \p^2_{z_n} S_{k-1}(y,z).
    \end{split}
\end{equation}

Notice that
\begin{equation}
    \begin{split}
    |S_{k-1}(y,z)| &\leq C (\ep + \delta^*_{k-1}), \qquad \text{for any }y,z\in D_0.
    \end{split}
\end{equation} 
To estimate the norms
\[
\|\gamma^{(1)}-\gamma^{(2)}\|_{L^{\infty}(D_k)} \quad \text{and}\quad \|q^{(1)}-q^{(2)}\|_{L^{\infty}(D_k)}
\]
one can proceed as in step $k=2$. Consider $\rho=r_0\slash 4$, $r\in(0,\bar{r}\slash 8)$ and set $w=P_{k}+r\nu(P_{k})$, then we split the integral solutions $S_{k-1}(w,w)$, $\p_{y_n}\p_{z_n}S_{k-1}(w,w)$ and $\p^2_{y_n}\p^2_{z_n} S_{k-1}(w,w)$ into the sum of two integrals over the domains $B_{\rho}(P_{k})\cap D_{k}$ and $\mathcal{U}_k\setminus(B_{\rho}(P_{k+1})\cap D_{k+1})$. At this point, one determines a lower bound for the intergral on the smallest domain and an upper bound for the integral on the largest domain using the estimates of Proposition \ref{proposizione unique continuation finale}. It turns out that
\begin{equation}\label{eqn: bound gammak}
     \|\gamma^{(1)}-\gamma^{(2)}\|_{L^{\infty}(D_k)} \le C (\ep + E) \omega^{(3k-4)}_{\eta_k} \Big(\frac{\ep}{\ep + E} \Big),
\end{equation}
and then by applying also \eqref{eqn: bound gammak},
\begin{equation}
     \|q^{(1)}-q^{(2)}\|_{L^{\infty}(D_k)} \le C (\ep + E) \omega^{(3(k-1))}_{\tilde \eta_k} \Big(\frac{\ep}{\ep + E} \Big),
\end{equation}
with $0<\eta_k, \tilde\eta_k<1$ constants that depend on the a-priori data only. 

Set $\bar{K}=\max\{ K, \tilde K\}$. Since $E=\delta^*_{\bar{K}}$, one derives
\[
E \leq C (\ep+E)\omega_{\tilde\eta_{\bar{K}}}^{(3(\bar{K}-1))}\left(\frac{\varepsilon}{\varepsilon+E}\right).
\]

If $E\geq e^2\ep$ (otherwise, the statement holds), it turns out that
\begin{equation}\label{eqn: lasteqn}
1 \leq C \omega_{\tilde\eta_{\bar{K}}}^{(3(\bar{K}-1))}\left(\frac{\ep}{E}\right).
\end{equation}
By applying the inverse of $\omega^{(3(\bar{K}-1))}_{\tilde\eta_{\bar{K}}}$ to \eqref{eqn: lasteqn}, we conclude that
\[
E \leq C_1\, \ep,
\] 
with $C_1$ a positive constant that depends on the a-priori data only. This concludes the proof of Theorem \ref{stability}.
\end{proof}

\section{Proof of the Auxiliary Propositions}\label{sec5}

To prove Proposition \ref{proposizione unique continuation finale}, we apply a result of propagation of smallness for elliptic PDEs with piecewise Lipschitz coefficients.

Let $\Omega\subset\R^n$ be a domain that satisfies the assumptions of Section \ref{sec: domain}.

First, we derive a three sphere inequality in terms of $L^{\infty}$ norms from the three sphere inequality demonstrated by \cite{Carstea2020}. 
\begin{lemma}\label{corollary appendix}
    Let $u\in H^1(B_{\rr})$ be a weak solution to
    \[
    \text{div}(\sigma\nabla u) + qu =0,\qquad \text{in }B_{\rr},
    \]
    with $B_{\rr} \subset (\Omega)_{r_0\slash 3}$, $\rr>0$. We assume that $\sigma,q$ satisfy the a-priori assumptions of Section \ref{sec: coefficient}. Then, for any $0<r_1<r_2<r_3\leq \rr$, the following inequality holds:
    \begin{equation}
        \|u\|_{L^{\infty}(B_{r_2})} \leq C_{\infty} \|u\|_{L^{\infty}(B_{r_1})}^{\beta} \|u\|_{L^{\infty}(B_{r_3})}^{1-\beta},
    \end{equation}
    where $\displaystyle \beta = \ln \Big(\frac{2r_3}{r_2+r_3}\Big)\slash\ln \Big(\frac{r_3}{r_1}\Big), \beta \in(0,1)$ and $C_{\infty}>1$ depends on $\bar{q}, \bar{A}, \frac{r_1}{r_2}, \frac{r_2}{r_3}, r_0, L, \lambda$.
\end{lemma}
\begin{proof}
The proof of this Lemma relies on the well-known $L^{\infty}-L^2$ Moser-Stampacchia estimates valid for elliptic equations with zeroth order term. By \cite[Theorem 8.17]{Gilbarg2001} (see also \cite[Theorem 6.1]{Brummelhuis1995}), there exists a constant $C>1$ that depends only on $\lambda, \bar{q}, \bar{A}, n$ such that, for any $0<r<\rho$,
    \begin{equation}\label{eqn: GT}
    \|u\|_{L^{\infty}(B_{r})} \leq \frac{C}{(\rho-r)^{\frac{n}{2}}} \|u\|_{L^2(B_{\rho})}.
    \end{equation}
    By \cite[Theorem 4.1]{Carstea2020} and \eqref{eqn: GT}, if we choose $r=r_2$, $\rho = (r_2+r_3)\slash 2$,
    \begin{equation*}
        \begin{split}
            \|u\|_{L^{\infty}(B_{r_2})} &\leq \frac{C}{\left(\frac{r_2+r_3}{2} - r_2\right)^{\frac{n}{2}}} \|u\|_{L^2(B_{\frac{r_2+r_3}{2}})}\\
            &\leq \frac{C}{\left(\frac{r_2+r_3}{2} - r_2\right)^{\frac{n}{2}}} \|u\|_{L^{2}(B_{r_1})}^{\beta} \|u\|_{L^{2}(B_{r_3})}^{1-\beta}\\
            &\leq \frac{C}{\left(\frac{r_2+r_3}{2} - r_2\right)^{\frac{n}{2}}} |B_{r_1}|^{\beta}  |B_{r_3}|^{\frac{1-\beta}{2}}\|u\|_{L^{\infty}(B_{r_1})}^{\beta} \|u\|_{L^{\infty}(B_{r_3})}^{1-\beta}.
        \end{split}
    \end{equation*}
\end{proof}

In the following Proposition we derive a result of propagation of smallness valid in our setting (see also \cite[Lemma 4.1]{Alessandrini2018} and \cite[Proposition 3.9]{Beretta2013}).

\begin{proposition}\label{preliminaryprop}
	For $k=0,\dots, N-1$ assume that there is a weak solution $v\in H^1(\mathcal{W}_k)$ to
	\begin{equation}\label{eq Wk}
		\mbox{div}(\sigma\,\nabla v) +qv =0 \qquad \mbox{in }\mathcal{W}_k.
	\end{equation}
	Suppose that for any given positive number $E_0,\varepsilon_0, \tilde \gamma$, the function $v$ satisfies
	\begin{equation}\label{first}
		|v(x)|\leq \ep_0 \qquad \text{for any } x\in D_0,
	\end{equation}
	and
	\begin{equation}\label{second}
		|v(x)|\leq C(E_0 +\ep_0)\big(r_0 d(x)\big)^{-\tilde\gamma}\qquad \text{for any } x\in \mathcal{W}_k,
	\end{equation}
    with $\tilde\gamma = n\slash 2-1$.
     Let $\bar{r}$ be the constant of Proposition \ref{corollary appendix}. Then, for any $r\in (0,\bar{r}\slash 4)$, there exist constants $C>1$ and $N_1\in \N$ such that
	\begin{equation}\label{eqn: singular}
		|v(y_{k+1})|\leq  C (E_0 + \ep_0)\Big(\frac{\ep_0}{\ep_0 + E_0}\Big)^{\tau_r \beta^{N_1}} r^{-\tilde \gamma},
	\end{equation}
	where $C, N_1$ depend on $r_0$, $L$, $\lambda$, $\bar{\sigma}, \bar{q}$ only, $y_{k+1} = P_{k+1} - r\nu(P_{k+1})$ with $\nu(P_{k+1})$ the outward unit normal of $\p D_k$ at $P_{k+1}$, $0<\beta<1$, $N_1\in\N$ and
    \[
    \tau_r = \ln\left(\frac{12r_1-2r}{12r_1-3r}\right)\slash \ln\left(\frac{6r_1-r}{2r_1} \right).
    \]
\end{proposition}

\begin{proof}[Proof of Proposition \ref{preliminaryprop}]
The proof follows the lines of \cite[Theorem 4.1]{Francini2023} and \cite[Proposition 3.9]{Beretta2013}. Let $P_0\in (D_0)_{r_0\slash 3}$, let $r_{00}>0$ be such that $B_{r_{00}} (P_0)\subset (D_0)_{r_0\slash 3}$. By \eqref{first}, 
\[
|v(x)| \leq \varepsilon_0\qquad \text{for any }x\in B_{r_{00}}(P_0).
\]
Let $\bar{y}_{k+1}= P_{k+1}-3\bar{r}\nu(P_{k+1})$, where $\nu(P_{k+1})$ is the outer unit normal of $\p D_{k}$ at $P_{k+1}$. For any point $y_0\in B_{r_{00}}(P_0)$ there exists a Jordan curve contained in $\mathcal{W}_k$ that joins $y_0$ to $\bar{y}_{k+1}$. Call this curve $c(t)\in C([0,1],\mathcal{W}_k)$, so that $c(0)=y_0$ and $c(1)=\bar{y}_{k+1}$. Let
\[
r_3=\frac{\bar{r}}{2}, \quad r_2 = \frac{3}{4}r_3,\quad r_1 = \frac{r_3}{4},
\]
so that $B_{r_1}(y_0)\subset B_{r_3}(y_0) \subset (D_0)_{r_0\slash 3}$. Define $0=t_0<t_1<\dots<t_{\bar{N}}=1$ so that 
\begin{align*}
    &t_{k+1} = \max\{t\,:\,|c(t)-c(t_k)|=2r_1\} \qquad\text{as long as }\quad |\bar{y}_{k+1} -c(t_k)|>2r_1\\
    &\text{otherwise }\bar{N}=k+1,\,\, t_{\bar{N}}=1.
\end{align*}
Notice that $B_{r_1}(c(t_k))\cap B_{r_1}(c(t_{k+1}))=\emptyset$ and $B_{r_1}(c(t_{k+1}))\subset B_{r_2}(c(t_k))$ for $k=1,\dots,\bar{N}-1$. Thanks to Lemma \ref{corollary appendix}, one can propagate the estimate $|v(y_0)|$ along the Jordan curve up to a ball centred at $\bar{y}_{k+1}$ of radius $r_1$ across the flat interfaces $\Sigma_m$ for $m\in \{1,\dots,k\}$. Hence, one derives
\[
|v(y_{k+1})| \leq C \ep_0^{\beta^{N_1}} (\ep_0 +E_0)^{1-\beta^{N_1}},
\]
with $0<\beta<1$, $N_1\in \N$ and $C>0$ depend on the a priori data only. 

Let $r<r_1$, $y_{k+1} =P_{k+1} -  r \nu(P_{k+1})$, then we can apply Lemma \ref{corollary appendix} to spheres centred at $\bar{y}_{k+1}$ of radii $r_1, 3r_1-r, 3r_1 - r\slash 2$ to obtain
\[
\|v\|_{L^{\infty}(B_{3r_1-r}(\bar{y}_{k+1})}\leq C r^{-(1-\tau_r)\tilde\gamma} \left( \frac{\ep_0}{\ep_0 + E_0}\right)^{\tau_r \beta^{N_1}} (\ep_0 +E_0)
\]
with
\[
\tau_r=\log\left(\frac{12r_1-2r}{12r_1-3r}\right)\slash \log\left(\frac{6r_1-r}{2r}\right).
\]
One derives \eqref{eqn: singular} by observing that
\[
C_1 r^{-\tilde\gamma} \leq r^{-(1-\tau_r)\tilde\gamma} \leq C_2 r^{-\tilde\gamma}.
\]
\end{proof}

 We are ready to prove the quantitative estimates of unique continuation for the singular solutions. 

 \begin{proof}[Proof of Proposition \ref{proposizione unique continuation finale}]
  Notice that for $y,z\in\mathcal{W}_k$,
  \[
  |S_k(y,z)|\leq C E (d(y)d(z))^{1-\frac{n}{2}},
  \]
  where $d(y)$ is the distance of $y$ from $\mathcal{U}_k$. 
  
 Fix $z\in (D_0)_{r_0\slash 3}$ and set $v(y)= S_k(y,z)$, then $v$ is a weak solution to 
 \[
 \nabla_y\cdot(\sigma^{(1)}\nabla_y v) + q^{(1)} v= 0, \qquad \text{in }\mathcal{W}_k.
 \]
 Notice that for $y\in\mathcal{W}_k$,
 \[
 |v(y)| \leq C E (d(y))^{1-\frac{n}{2}}.
 \]
 Thus, by applying Proposition \ref{preliminaryprop}, for $r\in (0,\bar{r}\slash 4)$, $z\in (D_0)_{r_0\slash 3}$ and $y_{k+1}= P_{k+1} - r\nu(P_{k+1})$, 
 \[
 |S_k(y_{k+1}, z)| \leq C r^{-\tilde\gamma} \left( \frac{\ep_0}{\ep_0 + E} \right)^{\tau_r \beta^{N_1}} (\ep_0 + E),
 \]
 with $\tilde \gamma = n\slash2 -1$. Now, set $\tilde v(z) = S_k(y_{k+1}, z)$ for $z\in \mathcal{W}_k$. Then $\tilde v$ is a weak solution to 
 \[
 \nabla_z\cdot(\sigma^{(2)}\nabla_z \tilde v) + q^{(2)} \tilde v = 0, \qquad \text{in }\mathcal{W}_k.
 \] 
 Since
 \[
 |\tilde v(z)| \leq C\,E\,(r\text{dist}(z,\Sigma_{k+1}))^{1-\frac{n}{2}},\qquad \text{for any }z\in \mathcal{W}_k,
 \]
 then
 \[
 |S_k(y_{k+1}, y_{k+1})| \leq C r^{-2\tilde\gamma}\left( \frac{\ep_0}{\ep_0 + E} \right)^{\tau_r^2 \beta^{2N_1}} (\ep_0 +E).
 \]
Let us determine the estimates for the partial derivatives of the integral solution. Since $S_k(y_1,\dots,y_n,z_1,\dots,z_n)$ is a weak solution to
 \[
 \nabla_y\cdot(\sigma^{(1)}\nabla_y S_k(y,z)) +\nabla_z\cdot(\sigma^{(2)}\nabla_z S_k(y,z)) + 
 q^{(1)} S_k(y,z) + q^{(2)} S_k(y,z) = 0, \qquad \text{in }D_k\times D_k,
 \]
one can apply the Schauder interior estimates (see \cite{Agmon1959} or \cite{stein1970}) at $y_{k+1}=P_{k+1} - 2r\nu(P_{k+1})$ and derive
\begin{equation*}
    \begin{split}
        \|\p_{y_j}\p_{z_i}S_k(y,z)\|_{L^{\infty}(B_{\frac{r}{2}}(y_{k+1})\times B_{\frac{r}{2}}(y_{k+1}))}\leq \frac{C}{r^2} \|S_k(y,z)\|_{L^{\infty}(B_{r}(y_{k+1})\times B_{r}(y_{k+1}))},
    \end{split}
\end{equation*}
and
\begin{equation*}
    \begin{split}
        \|\p^2_{y_j}\p^2_{z_i}S_k(y,z)\|_{L^{\infty}(B_{\frac{r}{4}}(y_{k+1})\times B_{\frac{r}{4}}(y_{k+1}))}\leq \frac{C}{r^2} \|\p_{y_j}\p_{z_i}S_k(y,z)\|_{L^{\infty}(B_{\frac{r}{2}}(y_{k+1})\times B_{\frac{r}{2}}(y_{k+1}))}.
    \end{split}
\end{equation*}
From the previous step, the thesis follows.
 \end{proof}

 \begin{proof}[Proof of Proposition \ref{prop: asymptotic estimates}]
    We prove \eqref{eqn: asymptotic3}. Fix $m\in\{0,\dots,N-1\}$ and let $Q_{m+1}\in \Sigma_{m+1}\cap B_{\frac{r_0}{4}}(P_{m+1})$. Up to a change of coordinates, we can assume that $Q_{m+1}$ coincides with the origin. Let $\gamma^+=\gamma_{m+1}(0)$, $\gamma^-=\gamma_{m}(0)$, $A=A(0)$, and define
    \[
    \sigma_0(x)=(\gamma^+\chi_{\R^n_+}(x) + \gamma^-\chi_{\R^n_-}(x))A.
    \]
    For simplicity, we write $y$ in place of $y_{m+1}$.
    Let $H$ be the fundamental solution associated to the elliptic operator div$(\sigma_0\nabla\cdot)$. For $y\in \Omega_0$, let $G(\cdot,y)$ be the weak solution to the boundary value problem \eqref{eqn: green system}. Define 
    \[
    R(x,y) := G(x,y) - H(x,y).
    \]
    For $y\in \Omega_0$, $R(\cdot,y)$ is a weak solution to
    \begin{equation*}
        \begin{cases}
            \mbox{div}(\sigma \nabla R(\cdot,y)) + q R(\cdot,y) = \mbox{div}((\sigma_0-\sigma)\nabla H(\cdot,y)) - q H(\cdot,y) &\text{in }\Omega_0,\\
            R(\cdot,y)=- H(\cdot,y) &\text{on }\p\Omega_0\setminus\Sigma_0,\\
            \sigma \nabla R(\cdot,y)\cdot\nu + i R(\cdot,y) = -\sigma \nabla H(\cdot,y)\cdot\nu - i H(\cdot,y) &\text{on }\Sigma_0.
        \end{cases}
    \end{equation*}
    By Green's identity, one derives
    \begin{equation*}
        \begin{split}
            R(x,y) &= -\int_{\Omega_0} (\sigma - \sigma_0)(z) \nabla_z H(z,y)\cdot \nabla_z G(z,x) \diff z +\int_{\Omega_0} H(z,y) q(z) G(z,x) \diff z +\\
            &+\int_{\p\Omega_0\setminus \Sigma_0} \sigma(z)\nabla_z G(z,x)\cdot\nu H(z,y) \diff S(z) - \int_{\Sigma_0} [\sigma_0 \nabla_z H(z,y)\cdot \nu + i H(z,y)] G(z,x) \diff S(z).
        \end{split}
    \end{equation*}
    Set $B=B_{r_0\slash 4}(Q_{m+1})$ and define
    \begin{equation*}
        \begin{split}
            \tilde R(x,y) &= \int_B (\sigma_0-\sigma)(z) \nabla_z H(z,y)\cdot \nabla_z G(z,x) \diff z +\int_{B} H(z,y) q(z) G(z,x) \diff z
        \end{split}
    \end{equation*}
    Since 
    \[
    |\nabla_y(R(x,y) - \tilde R(x,y))| \leq C,
    \]
    and
    \[
    |\nabla^2_y(R(x,y) - \tilde R(x,y))| \leq C,
    \]
    one has to study only the asymptotic behaviour of $\nabla_y\tilde R(x,y)$ and $\nabla_y^2 \tilde R(x,y)$. Let us prove an upper bound for $\nabla_y\tilde R(x,y)$. Set $B' = B'_{r_0\slash 4}$,
    \begin{align*}
            &B^+=\{x\in B\,:\,x_n>0\} \qquad B^-=\{x\in B\,:\,x_n<0\},\\
            &q^+=q|_{B^+},\quad q^-=q|_{B^-},\qquad[q]=(q^+-q^-)|_{B'},\\
            &\gamma^+=\gamma|_{B^+},\quad \gamma^-=\gamma|_{B^-},\qquad[\sigma]=(\sigma^+-\sigma^-)|_{B'}= (\gamma^+-\gamma-)|_{B'} A|_{B'}.
    \end{align*}
    It turns out that for $i=1,\dots,n$,
    \begin{equation}\label{eqn: firstremainder}
        \begin{split}
            \p_{y_i} &\tilde R(x,y) = \int_B \p_{z_i}((\sigma-\sigma_0)(z)\nabla_z H(z,y))\cdot \nabla_z G(z,x) \diff z - \int_B \p_{z_i} H(z,y) q(z) G(z,x) \diff z=\\
            &= \int_{\p B} (\sigma-\sigma_0)(z)\nabla_z H(z,y)\cdot \nabla_z G(z,x) e_i\cdot\nu \diff z - \int_{\p B} H(z,y) q(z) G(z,x) e_i\cdot\nu\diff z - \\
            &- \int_{B'} [(\sigma-\sigma_0)(z')] \nabla_z H(z',y)\cdot \nabla_z G(z',x) e_i\cdot e_n \diff z' + \int_{B'} H(z',y) [q(z')] G(z',x) e_i\cdot e_n \diff z' - \\
            &- \int_{B} (\sigma-\sigma_0)(z)\nabla_z H(z,y)\cdot \p_{z_i} \nabla_z G(z,x) \diff z + \int_{B} H(z,y) \p_{z_i}(q(z) G(z,x)) \diff z.         
        \end{split}
    \end{equation}
    Notice that $\p_{z_i}(\sigma-\sigma_0)(z)$ and $\p_{z_i}q(z)$ are well-defined on $B\setminus B'$. 
The first and second integrals on the right-hand side of \eqref{eqn: firstremainder} can be easily bounded by a positive constant that depends on the a priori data only. The fifth and sixth ones are dominated by
\begin{equation*}
    \begin{split}
        \int_B |(\sigma-\sigma_0)(z)|\,|\nabla_z H(z,y)|\,|\p_{z_i} \nabla_z G(z,x)| \diff z \leq C \int_B |z| |z-y|^{1-n} |z-x|^{-n} \leq C |x-y|^{1-n+\theta_3},
    \end{split}
\end{equation*}
with $\theta_3\in (0,1)$. Since $|x-y|^2 = |x_n + r|^2 + |x'|^2 \geq r^2$, one derives
\[
\int_B |(\sigma-\sigma_0)(z)|\,|\nabla_z H(z,y)|\,|\p_{z_i} \nabla_z G(z,x)| \diff z \leq C r^{1-n+\theta_3}.
\]
Notice that when $i\neq n$, the third and fourth integrals are equal to zero, hence
\[
|\p_{y_i} \tilde R(x,y)| \leq C |x-y|^{1-n+\theta_3}.
\]
When $i=j=n$, 
\begin{equation*}
    \begin{split}
        \Big|\int_{B'} [&(\sigma-\sigma_0)(z')] \nabla_z H(z',y)\cdot \nabla_z G(z',x) \diff z' +\\
            &+ \int_{B'} \p_{y_j} H(z',y) [q(z)] G(z',x) \diff z'\Big|\leq \\
            &\leq C \int_{B'} |z'| \cdot |z'-y|^{-n}\cdot |z'-x|^{1-n} \diff z\leq C |x-y|^{3-n-\alpha},
    \end{split}
\end{equation*}
with $0<\alpha<1$.
Hence we conclude that
\[
|\p_{y_n} \tilde R(x,y)| \leq C |x-y|^{1-n+\theta_3}, \qquad \text{with }\theta_3\in (0,1).
\]
The upper bound for $\nabla_y^2\tilde R(x,y)$ follows by similar computations. Indeed, by further differentiation, one derives
    \begin{equation}\label{eqn: second remainder}
        \begin{split}
            &\p_{y_j} \p_{y_i} \tilde R(x,y) = \int_{\p B} ((\sigma-\sigma_0)(z) \p_{y_j} \nabla_z H(z,y))\cdot \nabla_z G(z,x) e_i\cdot\nu \diff z - \\
            &-\int_{\p B} \p_{y_j} H(z,y) q(z) G(z,x) e_i\cdot\nu\diff z + \\
            &- \int_{B'} [(\sigma-\sigma_0)(z')] \p_{y_j} \nabla_z H(z',y)\cdot \nabla_z G(z',x) e_i\cdot e_n \diff z' +\\
            &+\int_{B'} \p_{y_j} H(z',y) [q(z')] G(z',x) e_i\cdot e_n \diff z' + \\
            &- \int_{B} \p_{y_j}(\sigma-\sigma_0)(z)\nabla_z H(z,y))\cdot \p_{z_i} \nabla_z G(z,x) \diff z + \\
            &+\int_{B} \p_{y_j} H(z,y) \p_{z_i}(q(z) G(z,x)) \diff z.          
        \end{split}
    \end{equation}
The first and second integrals on the righthand side of \eqref{eqn: second remainder} can be easily bounded. The fifth and sixth ones are dominated by
\begin{equation*}
    \begin{split}
        \Big|\int_B \p_{y_j}(\sigma-\sigma_0)(z)\,\nabla_z H(z,y)\,\p_{z_i} \nabla_z G(z,x) \diff z \Big| \leq C \int_B |z-y|^{1-n} |z-x|^{-n} \leq C |x-y|^{1-n}.
    \end{split}
\end{equation*}
Since $|x-y|^2 = |x_n + r|^2 + |x'|^2 \geq r^2$, one derives
\[
\Big|\int_B \p_{y_j}(\sigma-\sigma_0)(z)\,\nabla_z H(z,y)\,\p_{z_i} \nabla_z G(z,x) \diff z \Big| \leq C r^{1-n}.
\]
Notice that when $(i,j)\neq (n,n)$, the third and fourth integrals are equal to zero, hence
\[
|\p_{y_j} \p_{y_i} \tilde R(x,y)| \leq C |x-y|^{1-n}.
\]
When $i=j=n$, 
\begin{equation*}
    \begin{split}
        \Big|\int_{B'} [&(\sigma-\sigma_0)(z')] \p_{y_j} \nabla_z H(z',y)\cdot \nabla_z G(z,x) e_i\cdot e_n \diff z' +\\
            &+ \int_{B'} \p_{y_j} H(z',y) [q(z)] G(z',x) e_i\cdot e_n \diff z'\Big|\leq \\
            &\leq C \int_{B'} |z'| \cdot \frac{1}{|z'-y|^{n}}\cdot \frac{1}{|z'-x|^{1-n}} \diff z\leq C |x-y|^{2-n},
    \end{split}
\end{equation*}
Hence,
\[
|\p^2_{y_n} \tilde R(x,y)| \leq C |x-y|^{1-n}.
\]
\end{proof}

\section{Acknowledgements}
The work of SF was supported by the PRIN Grant No. 201758MTR2 and the INdAM GNAMPA project "Problemi inversi per equazioni alle derivate parziali e applicazioni" CUP\_E53C22001930001.

\appendix

\section{Appendix}

\begin{proof}[Proof of Theorem \ref{stability} (Stability at the boundary)]
Let $\{x_1,\dots,x_n\}$ be a coordinate system with origin at $P_1$. For any $y,z\in D_0$, the following identities hold:
\begin{equation}\label{eqna: green identity1}
    \begin{split}
        &\int_{\Sigma} [\sigma^{(2)}(x)\nabla_x G_2(x,z)\cdot \nu\, G_1(x,y) - \sigma^{(1)}(x) \nabla_x G_1(x,y)\cdot\nu\, G_2(x,z) ] \diff S(x) =\\
        &=\int_{\Omega} [(\sigma^{(1)}-\sigma^{(2)})(x) \nabla_x G_1(x,y)\cdot \nabla_x G_2(x,z) + (q^{(2)}-q^{(1)})(x) G_1(x,y) G_2(x,z)] \diff x,
    \end{split}
\end{equation}
and
\begin{equation}\label{eqna: 1der}
    \begin{split}
        &\int_{\Sigma} [\sigma^{(2)}(x)\nabla_x \p_{z_n} G_2(x,z)\cdot \nu\, \p_{y_n} G_1(x,y) - \sigma^{(1)}(x) \nabla_x \p_{y_n} G_1(x,y)\cdot\nu\, \p_{z_n} G_2(x,z) ] \diff S(x) =\\
        &=\int_{\Omega} [(\sigma^{(1)}-\sigma^{(2)})(x) \nabla_x \p_{y_n} G_1(x,y)\cdot \nabla_x \p_{z_n} G_2(x,z) + (q^{(2)}-q^{(1)})(x) \p_{y_n} G_1(x,y) \p_{z_n} G_2(x,z)] \diff x,
    \end{split}
\end{equation}
By \eqref{eqn: inequality Cauchy} and \eqref{eqna: green identity1}, 
 \begin{align}\label{eqna: bound S0}
\Big|\int_{\Sigma} [\sigma^{(2)}\nabla_x G_2(x,z)\cdot \nu\, G_1(x,y) - \sigma^{(1)} \nabla_x G_1(x,y)\cdot\nu\, G_2(x,z) ] \diff S(x)\Big| \leq C \ep\, (d(y) d(z))^{1-\frac{n}{2}},
\end{align}
where $d(y)$ denotes the distance between $y$ and $\Omega$. Notice that the norm $\|\gamma^{(1)}_1 - \gamma^{(2)}_1\|_{L^{\infty}(D_1)}$ can be evaluated in terms of the quantities
\begin{eqnarray*}
\|\gamma^{(1)}_1 - \gamma^{(2)}_1\|_{L^{\infty}(\Sigma_1\cap B_{\frac{r_0}{4}}(P_1))}\quad\mbox{and}\quad | \partial_{\nu}(\gamma^{(1)}_1 - \gamma^{(2)}_1)(P_1)|.
\end{eqnarray*}

Let $\rho=r_0\slash 4$, let $r\in(0,\bar{r}\slash 8)$ and set $w=P_1 + r\nu(P_1)$, where $\nu(P_1)$ is the outward unit normal of $\p D_1$ at $P_1$. Consider
\begin{equation}
    S_0(w,w) = I_1(w)+I_2(w),
\end{equation}
with
\begin{equation*}
    \begin{split}
        I_1(w) &= \int_{B_{\rho}(P_1)\cap D_1}(\gamma_1^{(1)}-\gamma_1^{(2)})(x)\:A(x)\,\nabla_x G_1(x,w)\cdot\nabla_x G_2(x,w) \diff x +\\
         &+\int_{B_{\rho}(P_1)\cap D_1}(q_1^{(2)}-q_1^{(1)})(x)\, G_1(x,w)\cdot G_2(x,w) \diff x,
    \end{split}
\end{equation*}
and
\begin{equation*}
    \begin{split}
        I_2(w) &= \int_{\Omega\setminus (B_{\rho}(P_1)\cap D_1)}(\sigma^{(1)}-\sigma^{(2)})(x)\,\nabla_x G_1(x,w)\cdot\nabla_x G_2(x,w) \diff x +\\
         &+\int_{\Omega \setminus (B_{\rho}(P_1)\cap D_1)}(q^{(2)}-q^{(1)})(x)\, G_1(x,w)\cdot G_2(x,w) \diff x.
    \end{split}
\end{equation*}
The volume integrals of $I_2(w)$ can be bounded from above via Caccioppoli inequality (see also \cite[Proposition 3.1]{Alessandrini2005}):
\begin{equation}\label{a1 stima I2}
    |I_2(w)|\leq C E \rho^{2-n}.
\end{equation}
	
Regarding $I_1(w)$, notice that there exists $x^*\in \overline{\Sigma_1\cap B_{\frac{r_0}{4}}(P_1)}$ such that
\begin{equation}\label{aexpansion1}
    (\gamma^{(1)}_1-\gamma^{(2)}_1)(x^*) = \|\gamma^{(1)}_1-\gamma^{(2)}_1\|_{L^{\infty}(\Sigma_1\cap B_{\frac{r_0}{4}}(P_1))}.
\end{equation}
By \eqref{aexpansion1},
\begin{equation*}
    \begin{split}
        I_1(w) &= \int_{B_{\rho}(P_1)\cap D_1}(\gamma_1^{(1)}-\gamma_1^{(2)})(x^*)\:A(x)\,\nabla_x G_1(x,w)\cdot\nabla_x G_2(x,w) \diff x +\\
        &+ \int_{B_{\rho}(P_1)\cap D_1} B_1 \cdot(x-x^*)\:A(x)\,\nabla_x G_1(x,w)\cdot\nabla_x G_2(x,w) \diff x +\\
         &+\int_{B_{\rho}(P_1)\cap D_1}(q_1^{(2)}-q_1^{(1)})(x)\, G_1(x,w)\cdot G_2(x,w) \diff x.
    \end{split}
\end{equation*}
By the asymptotic estimate \eqref{eqn: asymptotic1}, one obtains
\begin{align*}
    I_1(w) \geq &\|\gamma^{(1)}_1 - \gamma^{(2)}_1\|_{L^{\infty}(\Sigma_1\cap B_{\frac{r_0}{4}})} \Big\{ \int_{B_{\rho}(P_1)\cap D_1} A(x) \nabla_x H_1(x,w)\cdot \nabla_x H_2(x,w) \diff x -\\
    &-\int_{B_{\rho}(P_1)\cap D_1} |x-w|^{2(1-n)+\theta_1} \diff x - \int_{B_{\rho}(P_1)\cap D_1} |x-w|^{2(1-n +\theta_1)} \diff x\Big\} -\\
    &- C E \int_{B_{\rho}(P_1)\cap D_1} |x| |x-w|^{2(1-n)}\, \diff x - C E \int_{B_{\rho}(P_1)\cap D_1} |x-w|^{2(2-n)} \diff x.
\end{align*}
It turns out that
\begin{equation}\label{eqna: stabinterm1}
    |I_1(w)| \geq C \|\gamma^{(1)}_1-\gamma^{(2)}_1\|_{L^{\infty}(\Sigma_1\cap B_{\frac{r_0}{4}}(P_1))} r^{2-n} - C E r^{2-n+\theta_1} - C E r^{3-n}.
\end{equation}
If we rearrange the inequalities \eqref{eqna: stabinterm1} and \eqref{a1 stima I2} together with \eqref{eqna: bound S0}, we derive 
\begin{equation}\label{eqna: intermstab21}
    \begin{split}
    \|\gamma^{(1)}_1-\gamma^{(2)}_1&\|_{L^{\infty}(\Sigma_1\cap B_{\frac{r_0}{4}}(P_1))} r^{2-n} \le C E r^{3-n} + C E r^{2-n+\theta_1} +C \ep r^{2-n} + C E \rho^{2-n}.
    \end{split}
\end{equation}
Multiply \eqref{eqna: intermstab21} by $r^{n-2}$, then for $r\rightarrow 0^+$, 
\begin{equation}\label{eqna: intermstab3 1}
    \|\gamma^{(1)}_1-\gamma^{(2)}_1\|_{L^{\infty}(\Sigma_1\cap B_{\frac{r_0}{4}}(P_1))} \leq C \ep.
\end{equation}
A similar estimate can be derived for the derivative of $\gamma^{(1)}_1-\gamma^{(2)}_1$ along the normal direction $\nu$ at $P_1$ by means of an argument analogous to \cite[Theorem 2.3]{Alessandrini2017}. From Taylor's formula applied in a neighbourhood of the point $P_1$, one derives
\begin{align*}
(\gamma^{(1)}_1-\gamma^{(2)}_1)(x) &= (\gamma^{(1)}_1-\gamma^{(2)}_1)(P_1) + (D_T(\gamma^{(1)}_1-\gamma^{(2)}_1)(P_1))\cdot (x-P_1)' +\\
&+(\p_{\nu}(\gamma^{(1)}_1-\gamma^{(2)}_1)(P_1))\cdot (x-P_1)_n.
\end{align*}
Hence,
\begin{align*}
    &|\p_{y_n} \p_{z_n} S_0(w,w)| \geq\\  &\geq \Big|\int_{B_{\rho}(P_1)\cap D_1} \p_{\nu}(\gamma^{(1)}_1 - \gamma^{(2)}_1)(P_1)\cdot (x-P_1)_n A(x) \nabla_x \p_{y_n} G_1(x,w)\cdot \nabla_x \p_{z_n} G_2(x,w) \diff x \Big|-\\
    &- \Big|\int_{B_{\rho}(P_1)\cap D_1} D_T(\gamma^{(1)}_1 - \gamma^{(2)}_1)(P_1)\cdot (x-P_1)' A(x) \nabla_x \p_{y_n} G_1(x,w)\cdot \nabla_x \p_{z_n} G_2(x,w) \diff x \Big|-\\
    &- \Big|\int_{B_{\rho}(P_1)\cap D_1} (\gamma^{(1)}_1 - \gamma^{(2)}_1)(P_1) A(x) \nabla_x \p_{y_n} G_1(x,w)\cdot \nabla_x \p_{z_n} G_2(x,w) \diff x \Big|-\\
    &- \Big|\int_{B_{\rho}(P_1)\cap D_1} (q^{(2)}_1 - q^{(1)}_1)(x) \p_{y_n} G_1(x,w)\cdot \p_{z_n} G_2(x,w) \diff x \Big|\\
    &- \Big|\int_{\Omega\setminus (B_{\rho}(P_1)\cap D_1)} (\sigma^{(1)}-\sigma^{(2)})(x) \p_{y_n} \nabla_x G_1(x,w)\cdot \p_{z_n} \nabla_x G_2(x,w) \diff x \Big|-\\
    &- \Big|\int_{\Omega\setminus (B_{\rho}(P_1)\cap D_1)} (q^{(1)}-q^{(2)})(x) \p_{y_n} G_1(x,w)\cdot \p_{z_n} G_2(x,w) \diff x\Big|\\
    &= I_{11} - I_{12} - I_{13} - I_{14} - I_{15} - I_{16}.
\end{align*}
To estimate $I_{11}$ from below, we add and subtract the biphase fundamental solution and by \eqref{eqn: asymptotic2}, one derives
\begin{equation}
    I_{11}\geq C |\p_{\nu}(\gamma^{(1)}_1 - \gamma^{(2)}_1)(P_1)| r^{1-n} - C E r^{1-n+\theta_2}.
\end{equation}
To estimate the terms $I_{12}$ and $I_{13}$, notice that
\[
|(\gamma^{(1)}_1 - \gamma^{(2)}_1)(P_1)| + C |D_T(\gamma^{(1)}_1 - \gamma^{(2)}_1)(P_1)| \leq C \|\gamma^{(1)}_1 - \gamma^{(2)}_1\|_{L^{\infty}(\Sigma_1\cap B_{\frac{r_0}{4}})} \leq C \ep.
\]
Regarding the integral $I_{14}$, one bounds it from above as
\begin{align*}
    I_{14} &\leq \|q^{(2)}_1-q^{(1)}_1\|_{L^{\infty}(D_1)} \int_{D_1\cap B_{\rho}} |\p_{y_n} G_1(x,w)| |\p_{z_n} G_2(x,w)| \diff x\\
    &\leq C \int_{D_1\cap B_{\rho}} |x-w|^{2(1-n)} \leq C\,r^{2-n}.
\end{align*}
The integral $I_{15}$ and $I_{16}$ can be bounded by means of \cite[Proposition 3.1]{Alessandrini2005} as
\[
I_{15}, I_{16} \leq C E \rho^{-n}.
\]
To sum up, we have 
\begin{equation}
|\p_{\nu}(\gamma^{(1)}_1 - \gamma^{(2)}_1)(P_1)| r^{1-n} \leq |\p_{y_n} \p_{z_n} S_0(w,w)| + C \{ Er^{1-n+\theta_2} + \ep r^{-n} \}.
\end{equation}
Since 
\[
|\p_{y_n}\p_{z_n} S_0(w,w)| \leq C \ep r^{-n},
\]
one derives
\begin{equation}\label{eqn1a: interm}
|\p_{\nu}(\gamma^{(1)}_1 - \gamma^{(2)}_1)(P_1)| r^{1-n} \leq C \{ Er^{1-n+\theta_2} +  \ep r^{-n} \}.
\end{equation}
Multiply \eqref{eqn1a: interm} by $r^{n-1}$ to obtain
\[
|\p_{\nu}(\gamma^{(1)}_1 - \gamma^{(2)}_1)(P_1)| \leq C\{E r^{\theta_2} + \ep r^{-1}\}.
\]
By optimizing w.r.t. $r$, it turns out that
\begin{equation}
    |\p_{\nu}(\gamma^{(1)}_1-\gamma^{(2)}_1)(P_1)| \leq C(\varepsilon + E) \left(\frac{\varepsilon}{\varepsilon+E}\right)^{\frac{\theta_2}{\theta_2+1}}.
\end{equation}
and we set $\displaystyle\eta_1 = \frac{\theta_2}{\theta_2+1}.$ Hence, we conclude that
\begin{equation}\label{eqna: bound gamma boundary}
    \|\gamma^{(1)}_1-\gamma^{(2)}_1\|_{L^{\infty}(D_1)} \leq C(\varepsilon + E) \left(\frac{\varepsilon}{\varepsilon+E}\right)^{\eta_1}.
\end{equation}
\bigskip

\paragraph{Stability at the boundary for $q$}
Our goal is to derive a bound for $\|q^{(1)}_1-q^{(2)}_1\|_{L^{\infty}(D_1)}$ in terms of \eqref{eqna: bound gamma boundary}. Notice that the norm ${\|q^{(2)}_1-q^{(1)}\|_{L^{\infty}(D_1)}}$ can be evaluated in terms of the following quantities:
\begin{equation}\label{eqna: upper q1 stab}
\|q^{(2)}_1-q^{(1)}_1\|_{L^{\infty}(\Sigma_1\cap B_{\frac{r_0}{4}}(P_1))} \quad \text{and}\quad |\p_{\nu}(q^{(2)}_1 - q^{(1)}_1)(P_1)|.
\end{equation}
Let $\rho=\frac{r_0}{4}$, $r\in (0,\frac{\bar{r}}{8})$ and set $w = P_1 + r\,\nu(P_1)$. Consider
\[
\p_{y_n}\p_{z_n} S_0(w,w) = \p_{y_n}\p_{z_n} I_1(w) + \p_{y_n}\p_{z_n} I_2(w),
\]
with $w=P_1+r\nu(P_1)$, as above. The term $\p_{y_n}\p_{z_n} I_2(w)$ can be bounded from above as
\[
\p_{y_n}\p_{z_n} I_2(w) \leq C\,E \rho^{-n}.
\]
To determine a lower bound for $\p_{y_n}\p_{z_n} I_1(w)$, first notice that there exists a point $\bar{x}\in \overline{\Sigma_1\cap B_{\rho}(P_1)}$ such that
\[
(q^{(2)}_1-q^{(1)}_1)(\bar{x}) = \|q^{(2)}_1-q^{(1)}_1\|_{L^{\infty}(\Sigma_1\cap B_{\frac{r_0}{4}}(P_1))}.
\]
By \eqref{eqn: asymptotic1} and \eqref{eqna: bound gamma boundary} one derives
\begin{align*}
C \|q^{(2)}_1-q^{(1)}_1\|_{L^{\infty}(\Sigma_1\cap B_{\frac{r_0}{4}}(P_1))} r^{2-n} \leq |\p_{y_n}\p_{z_n} I_1(w)| + C E r^{2-n+\theta_1} + C(\varepsilon + E) \left(\frac{\varepsilon}{\varepsilon+E}\right)^{\eta_1} r^{-n}.
\end{align*}
By \eqref{eqna: 1der},
\[
|\p_{y_n}\p_{z_n} S_0(w,w)| \leq C\, \ep\, r^{-n},
\]
hence, if we collect the upper bound for $I_2(w)$ and the lower bound for $I_1(w)$, we derive
\begin{align*}
\|q^{(2)}_1-q^{(1)}_1\|_{L^{\infty}(\Sigma_1\cap B_{\frac{r_0}{4}}(P_1))} r^{2-n} &\leq C\Big\{\ep\, r^{-n}  + E r^{2-n+\theta_1} + (\varepsilon + E) \left(\frac{\varepsilon}{\varepsilon+E}\right)^{\eta_1} r^{-n} + E\Big\}.
\end{align*}
Multiply by $r^{n-2}$ to obtain 
\begin{align*}
    \|q^{(2)}_1-q^{(1)}_1\|_{L^{\infty}(\Sigma_1\cap B_{\frac{r_0}{4}}(P_1))} &\leq C (\varepsilon + E) \Big\{ \left(\frac{\varepsilon}{\varepsilon+E}\right)^{\eta_1} r^{-2} +  E r^{\theta_1}\Big\}.
\end{align*}
By optimizing with respect to $r$, one concludes that
\begin{equation}\label{eqn3a: bound q1}
    \|q^{(2)}_1-q^{(1)}_1\|_{L^{\infty}(\Sigma_1\cap B_{\frac{r_0}{4}}(P_1))} \leq C (E + \ep) \left( \frac{\ep}{\ep+E}\right)^{\frac{\eta_1 \theta_1}{\theta_1+2}}.
\end{equation}

To estimate $|\p_{\nu}(q^{(2)}_1 - q^{(1)}_1)(P_1)|$, consider the singular solution $\p^2_{y_i y_j}\p^2_{z_i z_j} S_0(w,w)$ and split it as the sum of the terms
\begin{align*}
I_1^{ij}(w) &= \int_{D_1\cap B_{\rho}(P_1)} (\sigma^{(1)}_1 - \sigma^{(2)}_1)(x) \nabla_x \p^2_{y_i y_j} G_1(x,w)\cdot \p^2_{z_i z_j} G_2(x,w) \diff x +\\
&+\int_{D_1\cap B_{\rho}(P_1)} (q^{(2)}_1 - q^{(1)}_1)(x) \p^2_{y_i y_j} G_1(x,w)\cdot \p^2_{z_i z_j} G_2(x,w) \diff x,
\end{align*}
and
\begin{align*}
I_2^{ij}(w) &= \int_{\Omega\setminus (D_1\cap B_{\rho}(P_1))} (\sigma^{(1)}-\sigma^{(2)})(x) \nabla_x \p^2_{y_i y_j} G_1(x,w)\cdot \nabla_x \p^2_{z_i z_j} G_2(x,w) \diff x +\\
&+ \int_{\Omega\setminus (D_1\cap B_{\rho}(P_1))} (q^{(2)} - q^{(1)})(x) \p^2_{y_i y_j} G_1(x,w)\cdot \p^2_{z_i z_j} G_2(x,w) \diff x.
\end{align*}
Set $I_m(w)=\{I^{ij}_m(w)\}_{i,j=1,\dots,n}$. Denote by $|I_m(w)|$ the Euclidean norm of the matrix $I_m(w)$. The upper bound for $|I_2(w)|$ is given by
\[
|I_2(w)| \leq C E \rho^{-(n+2)},
\]
where $C$ is a positive constant that depends on the a priori data only. For the lower bound for $I_1(w)$,  
\begin{align*}
    |I_1(w)| &\geq \frac{1}{n} \sum_{i,j=1}^n\Big\{\Big|\int_{D_1\cap B_{\rho}(P_1)} (\p_{\nu}(q^{(2)}_1-q^{(1)}_1)(P_1))\cdot (x-P_1)_n \p^2_{y_i y_j} G_1(x,w)\cdot \p^2_{z_i z_j} G_2(x,w) \diff x\Big| - \\
    &- \Big|\int_{D_1\cap B_{\rho}(P_1)} (D_T(q^{(2)}_1-q^{(1)}_1)(P_1))\cdot (x-P_1)' \p^2_{y_i y_j} G_1(x,w)\cdot \p^2_{z_i z_j} G_2(x,w) \diff x\Big|-\\
    &- \Big|\int_{D_1\cap B_{\rho}(P_1)} (q^{(2)}_1-q^{(1)}_1)(P_1) \p^2_{y_i y_j} G_1(x,w)\cdot \p^2_{z_i z_j} G_2(x,w) \diff x\Big|\Big\}-\\
    &- \Big|\int_{D_1\cap B_{\rho}(P_1)} (\sigma^{(2)}_1-\sigma^{(1)}_1)(x) \p^2_{y_i y_j}\nabla_x G_1(x,w)\cdot \p^2_{z_i z_j}\nabla_x G_2(x,w) \diff x\Big|.
\end{align*}
Since 
\[
|(q^{(2)}_1-q^{(1)}_1)(P_1)|+C |(D_T(q^{(2)}_1-q^{(1)}_1)(P_1))| \leq C\|q^{(2)}_1-q^{(1)}_1\|_{L^{\infty}(\Sigma_1\cap B_{\frac{r_0}{4}}(P_1))},
\]
by \eqref{eqn3a: bound q1} and \eqref{eqn: asymptotic3}, one derives
\begin{equation}\label{eqna: interm}
\begin{split}
|I_1(w)|\geq C |(\p_{\nu}(q^{(2)}_1-q^{(1)}_1)(P_1))| r^{1-n} - C (E + \ep) \left( \frac{\ep}{\ep + E}\right)^{\frac{\eta_1 \theta_1}{\theta_1+2}} r^{-n}-\\
- C E r^{1+\theta_2 -n} - C(\varepsilon + E) \left(\frac{\varepsilon}{\varepsilon+E}\right)^{\eta_1}r^{-2-n}.
\end{split}
\end{equation}
Since for $y,z\in (D_0)_{r_0\slash 3}$,
\begin{equation*}
    \begin{split}
        \int_{\Sigma} [&\sigma^{(2)}(x)\nabla_x \p^2_{z_n} G_2(x,z)\cdot \nu\, \p^2_{y_n} G_1(x,y) - \sigma^{(1)}(x) \nabla_x \p^2_{y_n} G_1(x,y)\cdot\nu\, \p^2_{z_n} G_2(x,z) ] \diff S(x) =\\
        &=\int_{\Omega} [(\sigma^{(1)}-\sigma^{(2)})(x) \nabla_x \p^2_{y_n} G_1(x,y)\cdot \nabla_x \p^2_{z_n} G_2(x,z) + (q^{(2)}-q^{(1)})(x) \p^2_{y_n} G_1(x,y) \p^2_{z_n} G_2(x,z)] \diff x,
    \end{split}
\end{equation*}
it turns out that
\begin{equation}\label{eqn3a: interm 2derS0}
|\p^2_{y_n}\p^2_{z_n} S_0(w,w)| \leq C\, \ep\, r^{-2-n}.
\end{equation}
By \eqref{eqna: interm} and \eqref{eqn3a: interm 2derS0}, one derives
\begin{equation*}
\begin{split}
    |(\p_{\nu}(q^{(2)}_1-q^{(1)}_1)(P_1))| r^{1-n} \leq C (E + \ep) \left( \frac{\ep}{\ep+E}\right)^{\frac{\eta_1 \theta_1}{\theta_1+2}} r^{-n} +\\
    + C(\varepsilon + E) \left(\frac{\varepsilon}{\varepsilon+E}\right)^{\eta_1}r^{-2-n}+C E r^{1+\theta_2 -n} C \ep r^{-1-n}.
    \end{split}
\end{equation*}
Multiply by $r^{n-1}$ the last equation and optimize with respect to $r$ leads to the estimate
\begin{equation*}
    |(\p_{\nu}(q^{(2)}_1-q^{(1)}_1)(P_1))| \leq C (E + \ep) \left( \frac{\ep}{\ep + E}\right)^{\eta_2},
\end{equation*}
with $\eta_2\in (0,1)$.

\end{proof}

\bibliographystyle{siam}
\bibliography{bibliography}

\begin{thebibliography}{10}

\bibitem{Agmon1959}
{\sc S.~Agmon, A.~Douglis, and L.~Nirenberg}, {\em Estimates near the boundary
  for solutions of elliptic partial differential equations satisfying general
  boundary conditions. {I}}, Comm. Pure Appl. Math., 12 (1959), pp.~623--727.

\bibitem{Alberti2023}
{\sc G.~S. Alberti, A.~Arroyo, and M.~Santacesaria}, {\em Inverse problems on
  low-dimensional manifolds}, Nonlinearity, 36 (2023), pp.~734--808.

\bibitem{Alessandrini1988}
{\sc G.~Alessandrini}, {\em Stable determination of conductivity by boundary
  measurements}, Appl. Anal., 27 (1988), pp.~153--172.

\bibitem{Alessandrini1990}
\leavevmode\vrule height 2pt depth -1.6pt width 23pt, {\em Singular solutions
  of elliptic equations and the determination of conductivity by boundary
  measurements}, J. Differential Equations, 84 (1990), pp.~252--272.

\bibitem{Alessandrini2019}
{\sc G.~Alessandrini, M.~V. de~Hoop, F.~Faucher, R.~Gaburro, and E.~Sincich},
  {\em Inverse problem for the {H}elmholtz equation with {C}auchy data:
  reconstruction with conditional well-posedness driven iterative
  regularization}, ESAIM Math. Model. Numer. Anal., 53 (2019), pp.~1005--1030.

\bibitem{Alessandrini2017}
{\sc G.~Alessandrini, M.~V. de~Hoop, R.~Gaburro, and E.~Sincich}, {\em
  Lipschitz stability for the electrostatic inverse boundary value problem with
  piecewise linear conductivities}, J. Math. Pures Appl. (9), 107 (2017),
  pp.~638--664.

\bibitem{Alessandrini2018}
\leavevmode\vrule height 2pt depth -1.6pt width 23pt, {\em Lipschitz stability
  for a piecewise linear {S}chr\"{o}dinger potential from local {C}auchy data},
  Asymptot. Anal., 108 (2018), pp.~115--149.

\bibitem{Alessandrini2005}
{\sc G.~Alessandrini and S.~Vessella}, {\em Lipschitz stability for the inverse
  conductivity problem}, Adv. in Appl. Math., 35 (2005), pp.~207--241.

\bibitem{Applegate2020}
{\sc M.~B. Applegate, R.~E. Istfan, S.~Spink, A.~Tank, and D.~Roblyer}, {\em
  Recent advances in high speed diffuse optical imaging in biomedicine}, APL
  Photonics, 5 (2020), p.~040802.

\bibitem{Arridge1999}
{\sc S.~R. Arridge}, {\em Optical tomography in medical imaging}, Inverse
  Problems, 15 (1999), pp.~R41--R93.

\bibitem{Arridge1998}
{\sc S.~R. Arridge and W.~R.~B. Lionheart}, {\em Nonuniqueness in
  diffusion-based optical tomography}, Opt. Lett., 23 (1998), pp.~882--884.

\bibitem{Arridge2009}
{\sc S.~R. Arridge and J.~C. Schotland}, {\em Optical tomography: forward and
  inverse problems}, Inverse Problems, 25 (2009), pp.~123010, 59.

\bibitem{aspri2022}
{\sc A.~Aspri, E.~Beretta, E.~Francini, and S.~Vessella}, {\em Lipschitz stable
  determination of polyhedral conductivity inclusions from local boundary
  measurements}, SIAM J. Math. Anal., 54 (2022), pp.~5182--5222.

\bibitem{Bellassoued2007}
{\sc M.~Bellassoued and M.~Yamamoto}, {\em Lipschitz stability in determining
  density and two {L}am\'{e} coefficients}, J. Math. Anal. Appl., 329 (2007),
  pp.~1240--1259.

\bibitem{Beretta2016}
{\sc E.~Beretta, M.~V. de~Hoop, F.~Faucher, and O.~Scherzer}, {\em Inverse
  boundary value problem for the {H}elmholtz equation: quantitative conditional
  {L}ipschitz stability estimates}, SIAM J. Math. Anal., 48 (2016),
  pp.~3962--3983.

\bibitem{Beretta2013}
{\sc E.~Beretta, M.~V. de~Hoop, and L.~Qiu}, {\em Lipschitz stability of an
  inverse boundary value problem for a {S}chr\"{o}dinger-type equation}, SIAM
  J. Math. Anal., 45 (2013), pp.~679--699.

\bibitem{Beretta2011}
{\sc E.~Beretta and E.~Francini}, {\em Lipschitz stability for the electrical
  impedance tomography problem: the complex case}, Comm. Partial Differential
  Equations, 36 (2011), pp.~1723--1749.

\bibitem{Beretta2022}
\leavevmode\vrule height 2pt depth -1.6pt width 23pt, {\em Global {L}ipschitz
  stability estimates for polygonal conductivity inclusions from boundary
  measurements}, Appl. Anal., 101 (2022), pp.~3536--3549.

\bibitem{Beretta2021}
{\sc E.~Beretta, E.~Francini, and S.~Vessella}, {\em Lipschitz stable
  determination of polygonal conductivity inclusions in a two-dimensional
  layered medium from the {D}irichlet-to-{N}eumann map}, SIAM J. Math. Anal.,
  53 (2021), pp.~4303--4327.

\bibitem{Brummelhuis1995}
{\sc R.~Brummelhuis}, {\em Three-spheres theorem for second order elliptic
  equations}, J. Anal. Math., 65 (1995), pp.~179--206.

\bibitem{Calderon1980}
{\sc A.-P. Calder\'{o}n}, {\em On an inverse boundary value problem}, in
  Seminar on {N}umerical {A}nalysis and its {A}pplications to {C}ontinuum
  {P}hysics ({R}io de {J}aneiro, 1980), Soc. Brasil. Mat., Rio de Janeiro,
  1980, pp.~65--73.

\bibitem{Carstea2020}
{\sc C.~I. C\^{a}rstea and J.-N. Wang}, {\em Propagation of smallness for an
  elliptic {PDE} with piecewise {L}ipschitz coefficients}, J. Differential
  Equations, 268 (2020), pp.~7609--7628.

\bibitem{Dicristo2003}
{\sc M.~Di~Cristo and L.~Rondi}, {\em Examples of exponential instability for
  inverse inclusion and scattering problems}, Inverse Problems, 19 (2003),
  pp.~685--701.

\bibitem{Eberle2021}
{\sc S.~Eberle, B.~Harrach, H.~Meftahi, and T.~Rezgui}, {\em Lipschitz
  stability estimate and reconstruction of {L}am\'{e} parameters in linear
  elasticity}, Inverse Probl. Sci. Eng., 29 (2021), pp.~396--417.

\bibitem{Foschiatti2021}
{\sc S.~Foschiatti, R.~Gaburro, and E.~Sincich}, {\em Stability for the
  {C}alder\'{o}n's problem for a class of anisotropic conductivities via an ad
  hoc misfit functional}, Inverse Problems, 37 (2021), pp.~Paper No. 125007,
  34.

\bibitem{Foschiatti2023}
{\sc S.~Foschiatti and E.~Sincich}, {\em Stable determination of an anisotropic
  inclusion in the {S}chr\"{o}dinger equation from local {C}auchy data},
  Inverse Probl. Imaging, 17 (2023), pp.~584--613.

\bibitem{Francini2023}
{\sc E.~Francini, S.~Vessella, and J.-N. Wang}, {\em Propagation of smallness
  and size estimate in the second order elliptic equation with discontinuous
  complex {L}ipschitz conductivity}, J. Differential Equations, 343 (2023),
  pp.~687--717.

\bibitem{Gaburro2015}
{\sc R.~Gaburro and E.~Sincich}, {\em Lipschitz stability for the inverse
  conductivity problem for a conformal class of anisotropic conductivities},
  Inverse Problems, 31 (2015), pp.~015008, 26.

\bibitem{Gebauer2008}
{\sc B.~Gebauer}, {\em Localized potentials in electrical impedance
  tomography}, Inverse Probl. Imaging, 2 (2008), pp.~251--269.

\bibitem{Gilbarg2001}
{\sc D.~Gilbarg and N.~S. Trudinger}, {\em Elliptic partial differential
  equations of second order}, Classics in Mathematics, Springer-Verlag, Berlin,
  2001.
\newblock Reprint of the 1998 edition.

\bibitem{Harrach2009}
{\sc B.~Harrach}, {\em On uniqueness in diffuse optical tomography}, Inverse
  Problems, 25 (2009), pp.~055010, 14.

\bibitem{Harrach2012}
\leavevmode\vrule height 2pt depth -1.6pt width 23pt, {\em Simultaneous
  determination of the diffusion and absorption coefficient from boundary
  data}, Inverse Probl. Imaging, 6 (2012), pp.~663--679.

\bibitem{harrach2023}
{\sc B.~Harrach and Y.-H. Lin}, {\em Simultaneous recovery of piecewise
  analytic coefficients in a semilinear elliptic equation}, Nonlinear Anal.,
  228 (2023), pp.~Paper No. 113188, 14.

\bibitem{Isakov1988}
{\sc V.~Isakov}, {\em On uniqueness of recovery of a discontinuous conductivity
  coefficient}, Comm. Pure Appl. Math., 41 (1988), pp.~865--877.

\bibitem{Isakov2017}
\leavevmode\vrule height 2pt depth -1.6pt width 23pt, {\em Inverse problems for
  partial differential equations}, vol.~127 of Applied Mathematical Sciences,
  Springer, Cham, third~ed., 2017.

\bibitem{Knyazev2010}
{\sc A.~Knyazev, A.~Jujunashvili, and M.~Argentati}, {\em Angles between
  infinite dimensional subspaces with applications to the {R}ayleigh-{R}itz and
  alternating projectors methods}, J. Funct. Anal., 259 (2010), pp.~1323--1345.

\bibitem{Mandache2001}
{\sc N.~Mandache}, {\em Exponential instability in an inverse problem for the
  {S}chr\"{o}dinger equation}, Inverse Problems, 17 (2001), pp.~1435--1444.

\bibitem{Rondi2003}
{\sc L.~Rondi}, {\em A remark on a paper by {G}. {A}lessandrini and {S}.
  {V}essella: ``{L}ipschitz stability for the inverse conductivity problem''
  [{A}dv. in {A}ppl. {M}ath. {\bf 35} (2005), no. 2, 207--241; mr2152888]},
  Adv. in Appl. Math., 36 (2006), pp.~67--69.

\bibitem{Ruland2019}
{\sc A.~R\"{u}land and E.~Sincich}, {\em Lipschitz stability for the finite
  dimensional fractional {C}alder\'{o}n problem with finite {C}auchy data},
  Inverse Probl. Imaging, 13 (2019), pp.~1023--1044.

\bibitem{Ruland2022}
\leavevmode\vrule height 2pt depth -1.6pt width 23pt, {\em On {R}unge
  approximation and {L}ipschitz stability for a finite-dimensional
  {S}chr\"{o}dinger inverse problem}, Appl. Anal., 101 (2022), pp.~3655--3666.

\bibitem{Sincich2007}
{\sc E.~Sincich}, {\em Lipschitz stability for the inverse {R}obin problem},
  Inverse Problems, 23 (2007), pp.~1311--1326.

\bibitem{stein1970}
{\sc E.~M. Stein}, {\em Singular integrals and differentiability properties of
  functions}, Princeton Mathematical Series, No. 30, Princeton University
  Press, Princeton, N.J., 1970.

\bibitem{Uhlmann1987}
{\sc J.~Sylvester and G.~Uhlmann}, {\em A global uniqueness theorem for an
  inverse boundary value problem}, Ann. of Math. (2), 125 (1987), pp.~153--169.

\bibitem{Uhlmann2014}
{\sc G.~Uhlmann}, {\em Inverse problems: seeing the unseen}, Bull. Math. Sci.,
  4 (2014), pp.~209--279.

\bibitem{vessella2023}
{\sc S.~Vessella}, {\em Notes on unique continuation properties for partial
  differential equations -- introduction to the stability estimates for inverse
  problems}, 2023.

\end{thebibliography}

\end{document}